\documentclass[11pt]{article}

\title{Reeb orbits frequently intersecting a symplectic surface}
\author{Michael Hutchings\footnote{Partially supported by NSF grant DMS-2404790.}}
\date{}

\usepackage{amssymb}
\usepackage{latexsym}
\usepackage{amsmath}
\usepackage{amsthm}
\usepackage{amscd}
\usepackage{color}

\numberwithin{equation}{section}

\newtheorem{theorem}{Theorem}[section]

\newtheorem{proposition}[theorem]{Proposition}
\newtheorem{corollary}[theorem]{Corollary}
\newtheorem{lemma}[theorem]{Lemma}
\newtheorem{lemma-definition}[theorem]{Lemma-Definition}

\theoremstyle{definition}
\newtheorem{definition}[theorem]{Definition}

\newtheorem{remark}[theorem]{Remark}

\newtheorem{example}[theorem]{Example}

\newtheorem{question}[theorem]{Question}

\newcommand{\C}{{\mathbb C}}
\newcommand{\Q}{{\mathbb Q}}
\newcommand{\R}{{\mathbb R}}

\newcommand{\Z}{{\mathbb Z}}

\newcommand{\op}{\operatorname}

\newcommand{\Ker}{\op{Ker}}

\newcommand{\bpm}{\begin{pmatrix}}
\newcommand{\epm}{\end{pmatrix}}

\renewcommand{\epsilon}{\varepsilon}

\begin{document}

\maketitle

\begin{abstract}
Consider a symplectic surface in a three-dimensional contact manifold with boundary on Reeb orbits (periodic orbits of the Reeb vector field). We assume that the rotation numbers of the boundary Reeb orbits satisfy a certain inequality, and we also make a technical assumption that the Reeb vector field has a particular ``nice'' form near the boundary of the surface. We then show that there exist Reeb orbits which intersect the interior of the surface, with a lower bound on the frequency of these intersections in terms of the symplectic area of the surface and the contact volume of the three-manifold.  No genericity of the contact form is assumed.  As a corollary of the main result, we obtain a generalization of various recent results relating the mean action of periodic orbits to the Calabi invariant for area-preserving surface diffeomorphisms.
\end{abstract}

\tableofcontents

\setcounter{tocdepth}{2}

\section{Introduction}

In recent years there has been substantial progress establishing properties of periodic orbits of Reeb vector fields (which we call ``Reeb orbits'') for $C^\infty$-generic contact forms on closed three-manifolds. In particular, Irie \cite{irieclosing} proved that for a $C^\infty$-generic contact form on a closed three-manifold $Y$, the set of Reeb orbits is dense in $Y$. This result was later refined to prove $C^\infty$-generic equidistribution of Reeb orbits \cite{irieequi} and a quantitative closing lemma \cite{altspec}. There is a similar story for periodic orbits of area-preserving surface diffeomorphisms \cite{ai,danclosing,pfh4,prasad}. There is also a parallel story, in a rather different geometric situation, for minimal hypersurfaces in Riemannian manifolds \cite{imn,mns}. In the above results, a key role is played by ``Weyl laws'' for ``spectral invariants'' in contact geometry and Riemannian geometry respectively \cite{vc,lmn}.

The above generic properties of Reeb orbits do not hold for all contact forms. For example, the standard contact form on the boundary of a four-dimensional irrational ellipsoid has only two simple Reeb orbits. In the present paper, we modify some arguments in the proofs of the above results to establish the existence of Reeb orbits with certain properties, namely Reeb orbits frequently intersecting a given symplectic surface, without any genericity hypothesis.

\subsection{Statement of the main result}

We will need the following definitions. Let $Y$ be a closed connected three-manifold, and let $\lambda$ be a contact form on $Y$. Let $\xi=\Ker(\lambda)\subset TY$ denote the associated contact structure. Let $R=R_\lambda$ denote the associated Reeb vector field characterized by $d\lambda(R,\cdot)=0$ and $\lambda(R)=1$.

A {\bf Reeb orbit\/} is a periodic orbit of $R$, namely a map $\gamma:\R/T\Z\to Y$ for some $T>0$, such that $\gamma'(t)=R(\gamma(t))$. We consider two such Reeb orbits to be equivalent if they differ by precomposition with a translation of the domain. The {\bf symplectic action\/} $\mathcal{A}(\gamma)=\mathcal{A}_\lambda(\gamma)$ is the period $T$, or equivalently the integral $\int_{\R/T\Z}\gamma^*\lambda$. The Reeb orbit $\gamma$ is {\bf simple\/} if the map $\gamma$ is an embedding. We sometimes identify a simple Reeb orbit with its image in $Y$. Let $\mathcal{P}(\lambda)$ denote the set of all simple Reeb orbits for $\lambda$.

If $\gamma:\R/T\Z\to Y$ is a Reeb orbit, the derivative of the time $T$ Reeb flow defines a symplectic linear map
\[
P_\gamma: (\xi_{\gamma(0)},d\lambda)\circlearrowleft.
\]
The Reeb orbit $\gamma$ is {\bf nondegenerate\/} if $1\notin\op{Spec}(P_\gamma)$. The contact form $\lambda$ is said to be nondegenerate if all Reeb orbits (including non-simple ones) are nondegenerate. The Reeb orbit $\gamma$ is {\bf elliptic\/} if $P_\gamma$ is conjugate to a rotation. If $\tau$ is a homotopy class of symplectic trivialization of $\gamma^*\xi$, then there is a well defined {\bf rotation number\/}
\[
\op{rot}_\tau(\gamma)\in\R.
\]
See e.g. \cite[\S3.2]{bn}. In the case where $\gamma$ is elliptic, one can homotope the trivialization $\tau$ so that for each $t\in\R$, the derivative of the time $t$ Reeb flow along $\gamma$ on $\xi$ is rotation by angle $2\pi t \cdot\op{rot}_\tau(\gamma)/T$.

We now want to consider Reeb orbits that intersect symplectic surfaces.

\begin{definition}
\label{def:ass}
An {\bf admissible symplectic surface\/} in $(Y,\lambda)$ is a map
\[
u:\Sigma\to Y
\]
such that:
\begin{description}
\item{(i)} $\Sigma$ is a compact oriented surface with boundary.
\item{(ii)} $u$ is an immersion, and $u|_{\op{int}(\Sigma)}$ is an embedding whose image is disjoint from $u(\partial\Sigma)$.
\item{(iii)} The Reeb vector field $R$ is positively transverse to $\op{int}(\Sigma)$.
\item{(iv)} If $B$ is a component of $\partial\Sigma$, then $u(B)$ is a simple Reeb orbit $\gamma$. If $B'$ is another component of $\partial\Sigma$ with $u(B')=\gamma$, then the maps $u|_B:B\to\gamma$ and $u|_{B'}:B'\to\gamma$ have the same degree.
\end{description}
\end{definition}

\begin{remark}
\label{rem:samedegree}
In condition (iv), the degree of the map $u|_B:B\to\gamma$ can be any nonzero integer. The last sentence in condition (iv) is somewhat redundant. Namely it holds up to sign by condition (ii); and if the ``niceness'' assumption in Definition~\ref{def:newnice} below holds, then the signs agree by condition (iii).
\end{remark}

\begin{remark}
In the terminology of \cite[\S2]{cdhr}, the map $u:\Sigma\to Y$ is a ``section for the flow'' of the Reeb vector field $R$. If one further assumes that every Reeb trajectory intersects $u(\Sigma)$ in forward and backward time (which we are not requiring in Definition~\ref{def:ass}), then $u$ is a ``Birkhoff section'' for the flow. It is shown in \cite{cdhr} that a $C^\infty$-generic contact form admits a Birkhoff section, and this gives many examples of admissible symplectic surfaces. 
\end{remark}

\begin{remark}
Under certain hypotheses, the projection of a finite energy holomorphic curve in $\R\times Y$ to $Y$ is an embedding, and then this gives rise to an admissible symplectic surface. See e.g. \cite[Prop.\ 3.3]{cghp} for one such set of hypotheses, based on ideas from \cite{hwz2,wendl}. More generally, when the projection of a holomorphic curve in $\R\times Y$ to $Y$ is not an embedding, one can modify it to obtain an admissible symplectic surface as explained in \cite[\S3.3]{cdr}, using ideas from \cite{fried}.
\end{remark}

\begin{definition}
If $u:\Sigma\to Y$ is an admissible symplectic surface, then it follows from condition (iii) in Definition~\ref{def:ass} that $u^*d\lambda>0$ on $\op{int}(\Sigma)$. We define the {\bf area\/}
\[
\op{Area}(\Sigma,d\lambda) = \int_\Sigma u^*d\lambda > 0.
\]
\end{definition}

\begin{definition}
\label{def:rotsigma}
Let $u:\Sigma\to Y$ be an admissible symplectic surface, and let $\gamma$ be a simple Reeb orbit in $u(\partial\Sigma)$. By Definition~\ref{def:ass}(iv), there is a nonzero integer $q$ such that each component of $\partial\Sigma$ that maps to $\gamma$ does so with degree $q$. Let $\tau$ be a symplectic trivialization of $\xi|_\gamma$, which we can identify with the normal bundle to $\gamma$ in $Y$. Then there is an integer $p_\tau$, relatively prime to $q$, such that with respect to $\tau$, the conormal direction to $\Sigma$ rotates $p_\tau$ times as one goes around a component of $\partial\Sigma$ in the positive Reeb direction. Define the ``rotation number of $\gamma$ with respect to $\Sigma$'' by
\[
\op{rot}_\Sigma(\gamma) = \op{rot}_\tau(\gamma) - \frac{p_\tau}{|q|}.
\]
\end{definition}

Note that $\op{rot}_\Sigma(\gamma)$ does not depend on the choice of trivialization $\tau$. One can think of $\op{rot}_\Sigma(\gamma)$ as the rotation number of $\gamma$ with respect to a fractional trivialization of $\xi|_\gamma$ determined by $\Sigma$. Since the interior of an admissible symplectic surface is positively transverse to the Reeb vector field, it follows that if $\op{rot}_\Sigma(\gamma) \neq 0$ (which is automatically true under the niceness condition in Definition~\ref{def:newnice} below), then the sign of $\op{rot}_\Sigma(\gamma)$ agrees with the sign of $q$ above.

\begin{definition}
If $\gamma$ is a simple Reeb orbit which is not in the image $u(\partial\Sigma)$, define $\gamma\cdot\Sigma$ to be the algebraic count of intersections of $\gamma$ with $\Sigma$. Note that each such intersection has multiplicity $+1$, so that $\gamma\cdot\Sigma\ge 0$.
\end{definition}

We now extend the above notion to Reeb orbits which are in the image $u(\partial\Sigma)$, in the spirit of \cite[\S5]{calabi}.

\begin{definition}
\label{def:gammadotsigma}
If $\gamma$ is a simple Reeb orbit in $u(\partial\Sigma)$, let $m$ denote the number of components of $\partial\Sigma$ mapping to $\gamma$, and let $q$ be as in Definition~\ref{def:rotsigma}. Define
\[
\gamma\cdot\Sigma = mq\cdot\op{rot}_\Sigma(\gamma) \ge 0.
\]
\end{definition}

The motivation for this definition is that if $k$ is a large positive integer, then a Reeb trajectory disjoint from $\gamma$ which stays near $\gamma$ and follows parallel to $\gamma$ approximately $k$ times will intersect $\op{int}(\Sigma)$ approximately $k(\gamma\cdot\Sigma)$ times.

In our main theorem we will need the following simplifying assumption on the triple $(Y,\lambda,u)$.

\begin{definition}
\label{def:newnice}
Let $u:\Sigma\to (Y,\lambda)$ be an admissible symplectic surface. We say that the triple $(Y,\lambda,u)$ is {\bf nice\/} if for each simple Reeb orbit $\gamma$ in $u(\partial\Sigma)$ with action $T$, there is a neighborhood $N(\gamma)$ of $\gamma$ and an orientation-preserving identification
\begin{equation}
\label{eqn:Ngammanew}
\psi: (\R/T\Z)\times D^2(r_0) \stackrel{\simeq}{\longrightarrow} N(\gamma)
\end{equation}
for some $r_0>0$ in which the Reeb vector field is given by
\begin{equation}
\label{eqn:modelReeb}
R = \partial_t + \frac{2\pi\rho}{T}\partial_\theta
\end{equation}
for some $\rho\in\R$. Here $t$ denotes the $\R/T\Z$ coordinate, and we use polar coordinates $r,\theta$ on $D^2(r_0)$. In particular, $\gamma$ is elliptic, and $\rho=\op{rot}_\tau(\gamma)$, where $\tau$ is the homotopy class of trivialization of $\gamma^*\xi$ determined by the derivative of \eqref{eqn:Ngammanew}.
\end{definition}

We can now state our main theorem. Roughly speaking, this asserts the existence of Reeb orbits with a certain lower bound on the frequency of intersections with $\Sigma$. Recall that the {\bf contact volume\/} is defined by
\[
\op{vol}(Y,\lambda) = \int_Y\lambda\wedge d\lambda > 0.
\]

\begin{theorem}
\label{thm:main1}
Let $Y$ be a closed connected three-manifold, let $\lambda$ be a contact form on $Y$, and let $u:\Sigma\to Y$ be an admissible symplectic surface. Suppose that $(Y,\lambda,u)$ is nice. Then
\begin{equation}
\label{eqn:boxed}
\boxed{
\sup_{\substack{\gamma\in\mathcal{P}(\lambda)}}\frac{\gamma\cdot\Sigma}{\mathcal{A}(\gamma)} \ge \frac{\op{Area}(\Sigma,d\lambda)}{\op{vol}(Y,\lambda)}.
}
\end{equation}
\end{theorem}

Theorem~\ref{thm:main1} guarantees for any $\epsilon>0$ the existence of a Reeb orbit $\gamma$ with
\begin{equation}
\label{eqn:guarantees}
\frac{\gamma\cdot\Sigma}{\mathcal{A}(\gamma)} > \frac{\op{Area}(\Sigma,d\lambda)}{\op{vol}(Y,\lambda)} - \epsilon.
\end{equation}
Note however that the Reeb orbit $\gamma$ might be contained in $u(\partial\Sigma)$. It could even be the case that there are no Reeb orbits at all intersecting $\op{int}(\Sigma)$; see Example~\ref{ex:ellipsoid2} below. If we make an additional hypothesis to rule out this possibility, then we obtain information about Reeb orbits that intersect $\op{int}(\Sigma)$. Namely, Theorem~\ref{thm:main1} is equivalent to the following slightly longer statement:

\begin{corollary}
\label{cor:main2}
Let $Y$ be a closed connected three-manifold, let $\lambda$ be a contact form on $Y$, and let $u:\Sigma\to Y$ be an admissible symplectic surface. Suppose that $(Y,\lambda,u)$ is nice. Suppose that for every simple Reeb orbit $\gamma\subset u(\partial\Sigma)$, we have
\begin{equation}
\label{eqn:infhyp}
\frac{\gamma\cdot\Sigma}{\mathcal{A}(\gamma)} < \frac{\op{Area}(\Sigma,d\lambda)}{\op{vol}(Y,\lambda)}.
\end{equation}
Then
\begin{equation}
\label{eqn:infbound}
\sup_{\substack{\gamma\in\mathcal{P}(\lambda)\\ \gamma\cap\op{int}(\Sigma)\neq\emptyset}}\frac{\gamma\cdot\Sigma}{\mathcal{A}(\gamma)} \ge \frac{\op{Area}(\Sigma,d\lambda)}{\op{vol}(Y,\lambda)}.
\end{equation}
\end{corollary}

The following examples illustrate the meaning of the numbers that appear in Corollary~\ref{cor:main2}.

\begin{example}
\label{ex:ellipsoid2}
Let $a,b>0$ be positive real numbers with $a/b\notin\Q$. Let $Y$ be the three-dimensional ellipsoid
\[
\partial E(a,b) = \left\{z\in\C^2\;\bigg|\; \frac{\pi|z_1|^2}{a} + \frac{\pi|z_2|^2}{b} = 1\right\}.
\]
Let $\lambda$ be the restriction to $Y$ of the standard Liouville form on $\R^4$ defined by
\begin{equation}
\label{eqn:standardLiouvilleform}
\lambda_0 = \frac{1}{2}\sum_{i=1}^2 \left(x_i\,dy_i - y_i\,dx_i\right).
\end{equation}
Then the Reeb vector field on $Y$ is given by
\begin{equation}
\label{eqn:ellipsoidR}
R = 2\pi\left(\frac{1}{a}\frac{\partial}{\partial\theta_1} + \frac{1}{b}\frac{\partial}{\partial\theta_2}\right)
\end{equation}
where $\theta_i$ denotes the argument of $z_i$. Thus there are just two simple Reeb orbits: the circle $\gamma_1=(\pi|z_1|^2=a,\;z_2=0)$, and the circle $\gamma_2=(z_1=0,\;\pi|z_2|^2=b)$. Their symplectic actions are
\begin{equation}
\label{eqn:ellipsoidaction}
\mathcal{A}(\gamma_1) = a, \quad \mathcal{A}(\gamma_2) = b.
\end{equation}
Also
\begin{equation}
\label{eqn:ellipsoidvolume}
\op{vol}(\partial E(a,b),\lambda) = ab.
\end{equation}

Now let $p,q>0$ be relatively prime positive integers. We can then form an admissible symplectic surface $u:\Sigma\to Y$ whose intersection with each torus in $Y$ where $|z_1|$ and $|z_2|$ are constant and nonzero is a line of constant slope in the coordinates $(\theta_1,\theta_2)$ representing the homology class $(p,q)$. (Compare \cite[\S1.4]{anchored}.) The triple $(Y,\lambda,u)$ is nice by \eqref{eqn:ellipsoidR}.

Suppose that $bq>ap$. Then the oriented boundary of $\Sigma$ is
\[
\partial\Sigma = q\gamma_2 - p\gamma_1,
\]
and
\begin{equation}
\label{eqn:ellipsoidarea}
\op{Area}(\Sigma,d\lambda) = bq - ap.
\end{equation}

Let $\tau$ be the homotopy class of trivialization of $\xi|_{\gamma_1}$ and $\xi|_{\gamma_2}$ given by the Seifert framings of $\gamma_1$ and $\gamma_2$. Then it follows from \eqref{eqn:ellipsoidR} and \eqref{eqn:ellipsoidaction} that
\[
\op{rot}_\tau(\gamma_1) = \frac{a}{b}, \quad \op{rot}_\tau(\gamma_2) = \frac{b}{a}.
\]
It follows from Definition~\ref{def:rotsigma} that
\[
\op{rot}_\Sigma(\gamma_1) = \frac{a}{b} - \frac{q}{p}, \quad \op{rot}_\Sigma(\gamma_2) = \frac{b}{a} - \frac{p}{q}.
\]
By Definition~\ref{def:gammadotsigma}, we have
\begin{equation}
\label{eqn:ellipsoidint}
\gamma_1\cdot\Sigma = \frac{bq-ap}{b}, \quad \gamma_2\cdot\Sigma = \frac{bq - ap}{a}.
\end{equation}
Combining \eqref{eqn:ellipsoidaction}, \eqref{eqn:ellipsoidvolume}, \eqref{eqn:ellipsoidarea}, and \eqref{eqn:ellipsoidint}, we conclude that the inequality \eqref{eqn:infhyp} does not hold for $\gamma_1$ or $\gamma_2$, but rather is an equality for both.
\end{example}

Corollary~\ref{cor:main2} does not tell us anything about the above example because the hypothesis \eqref{eqn:infhyp} is not satisfied, and in any case there are no Reeb orbits intersecting $\op{int}(\Sigma)$. 

\begin{example}
\label{ex:ellipsoid1}
To make a simpler example, let $Y$ be the ellipsoid as in Example~\ref{ex:ellipsoid2}, but now let $\Sigma$ be the disk in $Y$ where $z_2\ge 0$. This is a nice admissible symplectic surface with boundary $\gamma_1$. For this example we have
\[
\op{rot}_\Sigma(\gamma_1) = \gamma_1\cdot\Sigma = \frac{a}{b}.
\]
Also $\gamma_2\cdot\Sigma=1$ and $\op{Area}(\Sigma,d\lambda)=a$. Therefore
\[
\frac{\gamma_1\cdot\Sigma}{\mathcal{A}(\gamma_1)} = \frac{\gamma_2\cdot\Sigma}{\mathcal{A}(\gamma_2)} = \frac{1}{b} = \frac{\op{Area}(\Sigma,d\lambda)}{\op{vol}(Y,\lambda)}.
\]
So once again equality holds in \eqref{eqn:infhyp}, meaning that Corollary~\ref{cor:main2} is not applicable, but nonetheless $\gamma_2$ fulfills the conclusion \eqref{eqn:infbound} of Corollary~\ref{cor:main2}.
\end{example}

For a large family of examples where Corollary~\ref{cor:main2} leads to nontrivial conclusions, see \S\ref{sec:application} below. Meanwhile, to continue the discussion of the main theorem, let $\Phi:\R\times Y\to Y$ denote the flow of the Reeb vector field $R$.

\begin{remark}
The bound in \eqref{eqn:infbound} in Corollary~\ref{cor:main2} is the best possible without further hypotheses. The reason is that if $L>0$ satisfies $L\cdot\op{Area}(\Sigma,d\lambda) < \op{vol}(Y,\lambda)$, then the restriction of $\Phi$ to $[0,L]\times\op{int}(\Sigma)$ could be injective. Indeed this happens for all such $L$ in Examples~\ref{ex:ellipsoid2} and \ref{ex:ellipsoid1} above. In this case, a Reeb trajectory starting on $\op{int}(\Sigma)$ cannot return to $\Sigma$ in time $\le L$, so any Reeb orbit $\gamma$ intersecting $\op{int}(\Sigma)$ must satisfy $\mathcal{A}(\gamma)/(\gamma\cdot\Sigma) > L$. Compare \cite[Rmk.\ 1.14]{pfh4} and \cite[\S1.2]{altspec}.
\end{remark}

\begin{remark}
\label{rem:equi}
If $\lambda$ is $C^\infty$-generic, then in Corollary~\ref{cor:main2}, the inequality \eqref{eqn:infbound} holds without assuming niceness, the hypothesis \eqref{eqn:infhyp}, or that the boundary of $\Sigma$ maps to Reeb orbits, and we also have the reverse inequality
\begin{equation}
\label{eqn:supbound}
\inf_{\substack{\gamma\in\mathcal{P}(\lambda)\\ \gamma\not\subset u(\partial\Sigma)}}\frac{\gamma\cdot\Sigma}{\mathcal{A}(\gamma)} \le \frac{\op{Area}(\Sigma,d\lambda)}{\op{vol}(Y,\lambda)}.
\end{equation}
This follows from an equidistribution property of Reeb orbits proved by Irie \cite[Cor.\ 1.4]{irieequi}. Irie's result implies that if $\lambda$ is $C^\infty$-generic, then there is a sequence of simple Reeb orbits $(\gamma_k)_{k \ge 1}$ such that
if $\mathcal{U}\subset Y$ is an open set with measure zero boundary, then
\begin{equation}
\label{eqn:irie}
\lim_{k\to\infty}\frac{\mathcal{A}(\gamma_1\cap\mathcal{U})+\cdots+\mathcal{A}(\gamma_k\cap\mathcal{U})}{\mathcal{A}(\gamma_1)+\cdots+\mathcal{A}(\gamma_k)} = \frac{\op{vol}(\mathcal{U},\lambda)}{\op{vol}(Y,\lambda)}.
\end{equation}
Here $\mathcal{A}(\gamma_k\cap\mathcal{U})$ denotes the time spent by the Reeb trajectory $\gamma_k$ in $\mathcal{U}$.

To apply this result to a symplectic surface $\Sigma$ in $Y$, assume that $\lambda$ is $C^\infty$-generic in the above sense, let $\epsilon>0$ be sufficiently small that the restriction of the Reeb flow $\Phi$ to $(0,\epsilon)\times\op{int}(\Sigma)$ is injective\footnote{To guarantee that such $\epsilon>0$ exists one might need to first modify the surface near the boundary, cf.\ \S\ref{sec:rnc}.}, and take $\mathcal{U}=\Phi((0,\epsilon)\times\op{int}(\Sigma))$. Then it follows from \eqref{eqn:irie} that
\[
\lim_{k\to\infty}\frac{\gamma_1\cdot\Sigma+\cdots+\gamma_k\cdot\Sigma}{\mathcal{A}(\gamma_1)+\cdots+\mathcal{A}(\gamma_k)} = \frac{\op{Area}(\Sigma,d\lambda)}{\op{vol}(Y,\lambda)}.
\]
The inequalities \eqref{eqn:infbound} and \eqref{eqn:supbound} follow immediately.

It was shown in \cite[Thm.\ 1.9]{bechara} that \eqref{eqn:infbound} and \eqref{eqn:supbound} hold under slightly different hypotheses: namely Irie's equidistribution holds (which is true for $\lambda$ generic), $\partial\Sigma$ maps to Reeb orbits, and $\op{Area}(\Sigma,d\lambda)\neq 0$, but $\Sigma$ is not required to be symplectic.
\end{remark}

\begin{remark}
\label{rem:notdense}
In Theorem~\ref{thm:main1} and Corollary~\ref{cor:main2}, if $\Phi(\R\times\Sigma)$ is not dense in $Y$, then one can replace $\op{vol}(Y,\lambda)$ by the smaller number $\op{vol}\left(\overline{\Phi(\R\times\Sigma)},\lambda\right)$, which gives a stronger lower bound on the frequency. The reason is that one can define a contact form $\lambda'=e^{-f}\lambda$, where $f:Y\to[0,\infty)$ is supported in $Y\setminus\overline{\Phi(\R\times\Sigma)}$, such that $\op{vol}(Y,\lambda')$ is close to $\op{vol}\left(\overline{\Phi(\R\times\Sigma)},\lambda\right)$, and apply Theorem~\ref{thm:main1} to $\lambda'$.
\end{remark}


\subsection{Idea of the proof}

The proof of Theorem~\ref{thm:main1} uses the ``elementary spectral invariants'' $c_k(Y,\lambda)\in\R$ introduced in \cite{altspec}, and reviewed in \S\ref{sec:altspec}. These are related to the ECH spectral invariants defined in \cite{qech} but have a simpler definition. For any closed contact three-manifold $(Y,\lambda)$, the elementary spectral invariants are a sequence of real numbers
\[
0 = c_0(Y,\lambda) < c_1(Y,\lambda) \le c_2(Y,\lambda) \le \cdots < \infty
\]
satisfying various properties. One key property is ``spectrality'': each number $c_k(Y,\lambda)$ is the total action of some finite set of Reeb orbits. Another important property is that $c_k(Y,e^f\lambda)$ is a $C^0$-continuous function of $f$.  The most nontrivial property is the ``Weyl law''
\begin{equation}
\label{eqn:introweyl}
\lim_{k\to\infty}\frac{c_k(Y,\lambda)^2}{k}=2\op{vol}(Y,\lambda).
\end{equation}

To prove Theorem~\ref{thm:main1}, the idea is to define a one-parameter family of contact forms $\{\lambda_\delta\}_{\delta\ge 0}$ with $\lambda_0=\lambda$, where $\lambda_\delta$ is obtained from $\lambda$ by multiplying by the exponential of an appropriate function supported in a neighborhood of $\Sigma$. The contact form $\lambda_\delta$ is chosen so that, up to error terms that we omit here,
\begin{equation}
\label{eqn:volumeincrease}
\op{vol}(Y,\lambda_\delta) = \op{vol}(Y,\lambda) + 2\delta\cdot\op{Area}(\Sigma,d\lambda).
\end{equation}
In addition, a Reeb orbit $\gamma$ for $\lambda$ intersecting $\Sigma$ away from the boundary gives rise to a Reeb orbit $\gamma_\delta$ for $\lambda_\delta$ such that
\begin{equation}
\label{eqn:actionincrease}
\mathcal{A}_{\lambda_\delta}(\gamma_\delta) = \mathcal{A}_\lambda(\gamma) + \delta(\gamma\cdot\Sigma).
\end{equation}
The situation with Reeb orbits near the boundary of $\Sigma$ is more subtle, and we will ignore these Reeb orbits here; see Lemma~\ref{lem:slab} for the precise statement. (The ``niceness'' assumption in Definition~\ref{def:newnice} makes these subtleties more manageable.)

Now let $F$ be a real number and suppose that every Reeb orbit $\gamma$ for $\lambda$ satisfies
\begin{equation}
\label{eqn:suppose1}
\frac{\gamma\cdot\Sigma}{\mathcal{A}(\gamma)} \le F.
\end{equation}
Then from spectrality of $c_k$ and equation \eqref{eqn:actionincrease}, it follows that if $c_k(Y,\lambda_\delta)$ is a differentiable function of $\delta$ at $\delta=0$ (which is true for generic $\lambda$), then
\begin{equation}
\label{eqn:ifdiff}
\frac{d}{d\delta}\bigg|_{\delta=0}c_k(Y,\lambda_\delta) \le F\cdot c_k(Y,\lambda).
\end{equation}
On the other hand, the Weyl law \eqref{eqn:introweyl} for $\lambda_\delta$, combined with \eqref{eqn:volumeincrease}, gives
\[
c_k(Y,\lambda_\delta) = \sqrt{2k\left(
\op{vol}(Y,\lambda) + 2\delta\cdot\op{Area}(\Sigma,d\lambda)\right)} + \op{Error},
\]
where for a given $\delta$ the error term is $o(k^{1/2})$. As a heuristic, suppose that the above equation held without the error term. Then differentiating would give
\begin{equation}
\label{eqn:suppose2}
\frac{d}{d\delta}\bigg|_{\delta=0}c_k(Y,\lambda_\delta) = \frac{2k\cdot\op{Area}(\Sigma,d\lambda)}{c_k(Y,\lambda)} = \frac{\op{Area}(\Sigma,d\lambda)}{\op{vol}(Y,\lambda)} c_k(Y,\lambda).
\end{equation}
Combining \eqref{eqn:ifdiff} and \eqref{eqn:suppose2} would then give
\begin{equation}
\label{eqn:suppose3}
F\ge \frac{\op{Area}(\Sigma,d\lambda)}{\op{vol}(Y,\lambda)}.
\end{equation}
Thus \eqref{eqn:suppose1} would imply \eqref{eqn:suppose3}, proving Theorem~\ref{thm:main1}.

To convert the above heuristic into a proof, instead of differentiating the Weyl law, for suitable $\delta>0$ we will bound the difference $c_{k+l}(Y,\lambda_\delta) - c_k(Y,\lambda)$ from below in terms of the alternative ECH capacity $c_l^{\op{Alt}}$ (defined in \cite{altech} and reviewed in \S\ref{sec:altech}) of a four-dimensional region in the symplectization of $(Y,\lambda)$ determined by $\lambda_\delta$. Here $l$ is a certain multiple of $k$. If the conclusion of Theorem~\ref{thm:main1} fails, then similarly to \eqref{eqn:ifdiff}, this difference is too small when $k$ is large, contradicting the combination of the Weyl law for $c_k$ and another Weyl law for $c_l^{\op{Alt}}$.


\subsection{Questions for further research}

\begin{question}
One can use subleading asymptotics of spectral invariants, as in Remarks~\ref{rem:altechsubleading} and \ref{rem:altspecsubleading} below, to refine Theorem~\ref{thm:main1} to a quantitative statement, producing Reeb orbits $\gamma$ satisfying \eqref{eqn:guarantees} with an asymptotic upper bound on $\mathcal{A}(\gamma)$ in terms of $\epsilon$. What is the optimal bound?
\end{question}

\begin{question}
Can one remove the niceness hypothesis from Theorem~\ref{thm:main1}? If so, then Theorem~\ref{thm:main1} would admit many more examples. For example, for a reversible Finsler metric on a compact surface, this would imply the existence of geodesics that frequently intersect a given embedded geodesic. And for contact forms on connect sums as in \cite{luya}, this would imply the existence of Reeb orbits that frequently cross between the two halves of the connect sum.
\end{question}

\begin{question}
\label{q:main}
Under the hypotheses of Theorem~\ref{thm:main1}, does the reverse inequality
\[
\inf_{\substack{\gamma\in\mathcal{P}(\lambda) \\ \gamma\cap\Sigma\neq\emptyset}}\frac{\gamma\cdot\Sigma}{\mathcal{A}(\gamma)} \le \frac{\op{Area}(\Sigma,d\lambda)}{\op{vol}(Y,\lambda)}
\]
hold?
\end{question}

\begin{question}
Under the hypotheses of Theorem~\ref{thm:main1}, is the closure of the set
\[
\left\{\frac{\gamma\cdot\Sigma}{\mathcal{A}(\gamma)} \;\bigg|\;\gamma\in\mathcal{P}(\lambda), \;\gamma\cap\Sigma \neq \emptyset\right\} \subset \R
\]
an interval?
\end{question}

\begin{question}
Does Theorem~\ref{thm:main1} have any analogue for intersections of Reeb orbits with symplectic hypersurfaces in higher dimensional contact manifolds?
\end{question}


\paragraph{The rest of the paper.}

In \S\ref{sec:application} we explain an application of Theorem~\ref{thm:main1} to mean action of periodic orbits and the Calabi invariant for area-preserving surface diffeomorphisms. The results in that section are not used in the rest of the paper. In \S\ref{sec:spectral} we review and develop the alternative ECH capacities and elementary spectral invariants that we will need in the proof of Theorem~\ref{thm:main1}. In \S\ref{sec:proof} we prove Theorem~\ref{thm:main1}.

\paragraph{Acknowledgments.} I thank Umberto Hryniewicz, Kei Irie, Rohil Prasad, and Jun Zhang for helpful comments on previous versions of this paper. I thank the anonymous referee for many valuable corrections. ChatGPT was used for final proofreading.

\section{Application to mean action and the Calabi invariant}
\label{sec:application}

In this section, as an application of Corollary~\ref{cor:main2}, we generalize (in Theorem~\ref{thm:calabi} below) some previous results relating the mean action of periodic orbits of an area-preserving surface diffeomorphism to the Calabi invariant of the diffeomorphism.

To state the result, let $\Sigma$ be a compact connected surface with nonempty boundary, and fix a component $B$ of $\partial\Sigma$. Let $\omega$ be a symplectic form on $\Sigma$. Let $\phi:(\Sigma,\omega)\circlearrowleft$ be a symplectomorphism.

\begin{definition}
The {\bf flux\/} of $\phi$ is the class
\begin{equation}
\label{eqn:flux}
\op{Flux}(\phi) \in \frac{H^1(\Sigma;\R)}{\op{Im}(\phi^*-1:H^1(\Sigma;\R)\circlearrowleft)}
\end{equation}
defined as follows. Let $\beta$ be a primitive of $\omega$. Since $\phi$ is a symplectomorphism, $\phi^*\beta-\beta$ is a closed 1-form. We define
\begin{equation}
\label{eqn:fluxdef}
\op{Flux}(\phi) = [\phi^*\beta-\beta] \in H^1(\Sigma;\R).
\end{equation}
This gives a well defined class \eqref{eqn:flux}, because replacing $\beta$ with a different primitive $\beta'$ of $\omega$ will shift the right hand side of \eqref{eqn:fluxdef} by $(\phi^*-1)[\beta'-\beta]$.
\end{definition}

Let us now make the following assumptions on the symplectomorphism $\phi$:
\begin{description}
\item{(i)}
 $\op{Flux}(\phi)=0$.
\item{(ii)} $\int_\Sigma\omega = 1$.
\item{(iii)}
$\phi$ maps each component of $\partial\Sigma$ to itself, and $\phi$ is a rigid rotation near each component of the boundary. That is, if $Z$ is a component of $\partial\Sigma$, then there is a neighborhood $\mathcal{U}$ of $Z$ in $\Sigma$ with local coordinates $(r,\theta)\in(1-\epsilon,1]\times\R/2\pi\Z$ in which
\begin{equation}
\label{eqn:rigidrotation}
\phi(r,\theta)=(r,\theta+2\pi\theta_Z)
\end{equation}
for some real\footnote{Of course we only need an element of $\R/\Z$ to define the map \eqref{eqn:rigidrotation}, but below we will want to have chosen a specific lift of this number to $\R$.} number $\theta_Z$ depending only on the component $Z$.
\item{(iv)}
The sequence $\op{Ker}(\phi^*-1:H^1(\Sigma;\R)\circlearrowleft) \longrightarrow H^1(\partial\Sigma;\R) \stackrel{\int_{\partial\Sigma}}{\longrightarrow} \R$ is exact\footnote{This holds for example if $\Sigma$ has only one boundary component, or if $\Sigma$ has genus zero (thanks to assumption (iii)), or if $\phi$ is isotopic to the identity.}.
\end{description}

\begin{lemma}
\label{lem:primitive}
Under the above assumptions, there exists a primitive $\beta$ of $\omega$ such that $\phi^*\beta-\beta$ is exact; near the distinguished boundary component $B$ we have
\begin{equation}
\label{eqn:betaboundary1}
\beta = \frac{r^2\,d\theta}{2\pi};
\end{equation}
and near every other boundary component we have
\begin{equation}
\label{eqn:betaboundary2}
\beta = \frac{(r^2-1)d\theta}{2\pi},
\end{equation}
in local coordinates such that \eqref{eqn:rigidrotation} holds.
\end{lemma}

\begin{proof}
By (i), we can choose a primitive $\beta$ of $\omega$ such that $\phi^*\beta-\beta$ is exact. By (ii), we have $\int_{\partial\Sigma}\beta=1$. By (iv), by adding a closed $1$-form to $\beta$, we can arrange that $\beta$ integrates to $1$ on $B$ and to $0$ on every other component of $\partial\Sigma$, while maintaining the exactness of $\phi^*\beta-\beta$. By adding to $\beta$ an exact $1$-form supported in a neighborhood of $\partial\Sigma$, we can further arrange that the boundary conditions \eqref{eqn:betaboundary1} and \eqref{eqn:betaboundary2} hold.
\end{proof}

Let $\beta$ be a primitive of $\omega$ given by Lemma~\ref{lem:primitive}.
There is then a unique function $f_\beta:\Sigma\to\R$ such that
\begin{equation}
\label{eqn:fbeta}
\begin{split}
df_\beta &= \phi^*\beta - \beta,\\
{f_\beta}|_B & \equiv \theta_B.
\end{split}
\end{equation}
Note that $f_\beta$ is constant on a neighborhood of each component of $\partial\Sigma$.

\begin{definition}
\label{def:calabi}
Define the {\bf Calabi invariant\/}
\[
\op{Cal}(\phi,\theta_B,\beta) = \int_\Sigma f_\beta\omega \in \R.
\]
\end{definition}

Next, a {\bf periodic orbit\/} of $\phi$ is a cycle $\gamma = (x_1,\ldots,x_{d(\gamma)})$ of points in $\Sigma$ (which we will assume are distinct) with $\phi(x_i)=x_{i+1\mod d(\gamma)}$ for each $i=1,\ldots,d(\gamma)$. Let $\mathcal{P}(\phi)$ denote the set of periodic orbits of $\phi$.

\begin{definition}
\label{def:action}
If $\gamma=(x_1,\ldots,x_{d(\gamma)})$ is a periodic orbit of $\phi$, define the {\bf action\/}
\[
\mathcal{A}_\beta(\gamma) = \sum_{i=1}^{d(\gamma)} f_\beta(x_i).
\]
The {\bf mean action\/} of $\gamma$ is $\mathcal{A}_\beta(\gamma)/d(\gamma)$.
\end{definition}

We now examine the dependence of the mean action and the Calabi invariant on the choice of $\beta$.

\begin{lemma}
\label{lem:exactdifference}
If $\beta'$ is another primitive of $\omega$ such that $\beta'-\beta$ is exact and $\beta'$ also satisfies the boundary conditions \eqref{eqn:betaboundary1} and \eqref{eqn:betaboundary2}, then:
\begin{description}
\item{(a)}
 $\op{Cal}(\phi,\theta_B,\beta')=\op{Cal}(\phi,\theta_B,\beta)$.
 \item{(b)}
 If $\gamma=(x_1,\ldots,x_{d(\gamma)})$ is a periodic orbit of $\phi$, then
 $\mathcal{A}_{\beta'}(\gamma) = \mathcal{A}_\beta(\gamma)$.
 \end{description}
\end{lemma}

\begin{proof}
(cf.\ \cite{gg}) Write $\beta'=\beta+d\mu$ where $\mu:\Sigma\to\R$.
By \eqref{eqn:fbeta} we have
\begin{equation}
\label{eqn:fbeta'fbeta}
f_{\beta'}-f_\beta = \phi^*\mu-\mu.
\end{equation}

To prove (a), we compute
\[
\int_\Sigma f_{\beta'}\omega - \int_\Sigma f_\beta\omega = \int_\Sigma(\phi^*\mu)\omega - \int_\Sigma\mu\omega.
\]
The right hand side is zero because $\phi$ is a symplectomorphism.

To prove (b), by \eqref{eqn:fbeta'fbeta} we have a telescoping sum
\[
\mathcal{A}_{\beta'}(\gamma) - \mathcal{A}_\beta(\gamma) =
\sum_{i=1}^{d(\gamma)} (f_{\beta'}- f_\beta)(x_i) = \sum_{i=1}^{d(\gamma)}(\mu(x_{\mbox{\scriptsize $i+1$ mod $d(\gamma)$}}) - \mu(x_i)) = 0.
\]
\end{proof}

\begin{remark}
\label{rem:betaindependence}
If $\phi^*-1:H^1(\Sigma;\R)\circlearrowleft$ is an isomorphism, then any two valid choices of $\beta$ differ by an exact one-form. In this case it follows from Lemma~\ref{lem:exactdifference} that the Calabi invariant and the action do not depend on the choice of $\beta$.
\end{remark}

\begin{theorem}
\label{thm:calabi}
Let $\phi:(\Sigma,\omega)\circlearrowleft$ be a surface symplectomorphism satisfying conditions (i)--(iv) above for some $\theta_B\in\R$. Let $\beta$ be a primitive of $\omega$ such that $\phi^*\beta-\beta$ is exact and the boundary conditions \eqref{eqn:betaboundary1} and \eqref{eqn:betaboundary2} hold.  Suppose further that
\begin{equation}
\label{eqn:calabihyp}
\op{Cal}(\phi,\theta_B,\beta) < \min_{x\in\partial\Sigma}f_\beta(x).
\end{equation}
Then
\begin{equation}
\label{eqn:calabiinf}
\inf_{\gamma\in\mathcal{P}(\phi)}\frac{\mathcal{A}_\beta(\gamma)}{d(\gamma)} \le \op{Cal}(\phi,\theta_B,\beta).
\end{equation}
\end{theorem}

One way to think about this result is that the mean action $\mathcal{A}_\beta(\gamma)/d(\gamma)$ is the average of $f_\beta$ over the periodic orbit $\gamma$, while the Calabi invariant $\op{Cal}(\phi,\theta_B,\beta)$ is the average of $f_\beta$ over the surface $\Sigma$. Thus \eqref{eqn:calabiinf} asserts that to some extent, the average of $f_\beta$ over the surface is seen by periodic orbits.

\begin{remark}
The special case of Theorem~\ref{thm:calabi} where $\Sigma$ is a disk was proved in \cite{calabi}. Weiler \cite{weiler} proved relations between mean action and the Calabi invariant on an annulus, including some cases to which Theorem~\ref{thm:calabi} does not apply; see \cite[\S1.2]{weiler}. Nelson-Weiler \cite[Thm.\ 1.16]{nw} proved a version of Theorem~\ref{thm:calabi} where $\Sigma$ has genus $(p-1)(q-1)/2$ and one boundary component, and $\phi$ is isotopic to the monodromy map associated to a $(p,q)$ torus knot. All of these proofs involved computing a filtration on embedded contact homology (ECH) defined using linking number with a transverse knot \cite{calabi}. Our proof of Theorem~\ref{thm:calabi} does not mention the knot filtration on ECH and thus seems to be independent of the above proofs.
\end{remark}

\begin{remark}
Pirnapasov \cite{pirnapasov} proved a generalization of the disk case of Theorem~\ref{thm:calabi}, in which $\phi$ is not required to be a rigid rotation on or near the boundary, by modifying the map near the boundary to reduce to the rigid rotation case. It is also possible to extend Theorem~\ref{thm:calabi} in the same way, but we will omit this extension for brevity. Bramham-Pirnapasov \cite{bp} further extended the disk case to remove the hypothesis \eqref{eqn:calabihyp}, using Viterbo's spectral invariants \cite{viterbo}. We do not know to what extent this hypothesis can be removed for more general surfaces; cf.\ Question~\ref{q:main} and Remark~\ref{rem:ged} below.
\end{remark}

\begin{remark}
\label{rem:ged}
A generic equidistribution result of Pirnapasov-Prasad \cite[Thm.\ 1.6]{pp} implies, similarly to Remark~\ref{rem:equi}, that a version of Theorem~\ref{thm:calabi} (in which $\phi$ is not assumed to be a rigid rotation near each boundary component) holds for $C^\infty$-generic $\phi$ satisfying the zero flux condition, without the hypothesis \eqref{eqn:calabihyp}. This is explained for the case of Hamiltonian diffeomorphisms in \cite[Thm.\ 1.12]{pp}. It is not possible to simultaneously remove the genericity hypothesis and remove the hypothesis \eqref{eqn:calabihyp}, as shown by the example of an irrational rotation of an annulus, which has no periodic orbits. 
\end{remark}

Similarly to \cite[Cor.\ 1.5]{calabi}, we also have a ``dual'' result:

\begin{corollary}
\label{cor:inverse}
Let $\phi:(\Sigma,\omega)\circlearrowleft$ be a surface symplectomorphism satisfying conditions (i)--(iv) above for some $\theta_B\in\R$. Let $\beta$ be a primitive of $\omega$ such that $\phi^*\beta-\beta$ is exact and the boundary conditions \eqref{eqn:betaboundary1} and \eqref{eqn:betaboundary2} hold.  Suppose further that
\[
\op{Cal}(\phi,\theta_B,\beta) > \max_{x\in\partial\Sigma}f_\beta(x).
\]
Then
\[
\sup_{\gamma\in\mathcal{P}(\phi)}\frac{\mathcal{A}_\beta(\gamma)}{d(\gamma)} \ge \op{Cal}(\phi,\theta_B,\beta).
\]
\end{corollary}

\begin{proof}[Proof of Corollary~\ref{cor:inverse}, assuming Theorem~\ref{thm:calabi}]
Replace $(\phi,\theta_B)$ by $(\phi^{-1},-\theta_B)$, but keep the same primitive $\beta$. This replaces $f_\beta$ by $-f_\beta\circ\phi^{-1}$.
\end{proof}

\begin{proof}[Proof of Theorem~\ref{thm:calabi}, assuming Corollary~\ref{cor:main2}.]
We proceed in two steps.

{\em Step 1.\/} We first prove the theorem in the special case when $\Sigma$ has only one boundary component, so that $\partial\Sigma=B$.

We can assume without loss of generality that $f_\beta$ is everywhere positive, by adding a large integer $n$ to $\theta_B$, as this will increase $f_\beta$, the Calabi invariant, and the mean action of periodic orbits all by $n$.

By\footnote{\cite[Prop.\ 2.1]{calabi} is stated for a disk, but the proof in \cite{calabi} also works for a surface of arbitrary genus. \cite[Prop.\ 2.1]{calabi} does not mention niceness, but this follows from the construction.} \cite[Prop.\ 2.1]{calabi}, there exists a closed connected three-manifold $Y$, and a contact form $\lambda$ on $Y$, such that:
\begin{itemize}
\item
There is a nice embedded admissible symplectic surface $u:\Sigma\to Y$ for which
\begin{equation}
\label{eqn:calabiarea}
\op{Area}(\Sigma,d\lambda)=1,
\end{equation}
or equivalently $\mathcal{A}_\lambda(\partial\Sigma)=1$, and
\begin{equation}
\label{eqn:calabirot}
\op{rot}_\Sigma(\partial\Sigma) = \theta_B^{-1}.
\end{equation}
\item
There is a bijection $\mathcal{P}(\lambda) = \{\partial\Sigma\}\cup\mathcal{P}(\phi|_{\op{int}(\Sigma)})$, such that if $\gamma\in\mathcal{P}(\phi|_{\op{int}(\Sigma)})$, then the corresponding simple Reeb orbit of $\lambda$ has $d(\gamma)$ intersections with $\Sigma$ and symplectic action $\mathcal{A}_\beta(\gamma)$.
\item
We have
\begin{equation}
\label{eqn:calabivol}
\op{vol}(Y,\lambda) = \op{Cal}(\phi,\theta_B,\beta).
\end{equation}
\end{itemize}

It follows from the hypothesis \eqref{eqn:calabihyp} and equations \eqref{eqn:calabiarea}--\eqref{eqn:calabivol} that $(Y,\lambda,u)$ satisfies the hypothesis \eqref{eqn:infhyp} of Corollary~\ref{cor:main2}. By \eqref{eqn:calabiarea}, \eqref{eqn:calabivol}, and the second bullet point above, the conclusion \eqref{eqn:infbound} of Corollary~\ref{cor:main2} gives the desired inequality \eqref{eqn:calabiinf}.

{\em Step 2.\/} We now prove the theorem in the general case. To do so, we form a new symplectic surface $(\widehat{\Sigma},\widehat{\omega})$ with one boundary component by collapsing each component other than $B$ to a point. Since we assumed that $\phi$ is a rigid rotation near each boundary component, it follows that $\phi$ descends to a smooth symplectomorphism $\widehat{\phi}$ of $(\widehat{\Sigma},\widehat{\omega})$. The periodic orbits of $\widehat{\phi}$ consist of the periodic orbits of $\phi$ not contained in boundary components other than $B$, together with fixed points $z_1,\ldots,z_k$ arising from the collapsed boundary components.

It follows from the boundary condition \eqref{eqn:betaboundary2} that $\beta$ descends to a primitive $\widehat{\beta}$ of $\widehat{\omega}$. Then $f_\beta$ is the pullback of $f_{\widehat{\beta}}$ under the quotient map $\Sigma\to\widehat{\Sigma}$. In particular,
\begin{equation}
\label{eqn:caldescend}
\op{Cal}(\widehat{\phi},\theta_B,\widehat{\beta}) = \op{Cal}(\phi,\theta_B,\beta),
\end{equation}
and
\begin{equation}
\label{eqn:actiondescend}
\mathcal{A}_{\widehat{\beta}}(\gamma) = \mathcal{A}_\beta(\gamma)
\end{equation}
when $\gamma$ is a periodic orbit of $\widehat{\phi}$ which is not one of the fixed points $z_1,\ldots,z_k$.

By the hypothesis \eqref{eqn:calabihyp}, the special case of the theorem proved in Step 1 is applicable and gives
\begin{equation}
\label{eqn:nfri}
\inf_{\gamma\in\mathcal{P}(\widehat{\phi})}\frac{\mathcal{A}_{\widehat{\beta}}(\gamma)}{d(\gamma)} \le \op{Cal}(\widehat{\phi},\theta_B,\widehat{\beta}).
\end{equation}
By the hypothesis \eqref{eqn:calabihyp}, none of the fixed points $z_1,\ldots,z_k$ of $\widehat{\phi}$ realizes the infimum in \eqref{eqn:nfri}. Thus \eqref{eqn:caldescend}, \eqref{eqn:actiondescend}, and \eqref{eqn:nfri} imply \eqref{eqn:calabiinf}.
\end{proof}

\section{Spectral invariants}
\label{sec:spectral}

We now prepare for the proof of Theorem~\ref{thm:main1}. We will need two different families of ``spectral invariants'': the ``alternative ECH capacities'' \cite{altech} of four-dimensional symplectic manifolds, which we denote by $c_k^{\op{Alt}}$, and the ``elementary spectral invariants'' \cite{altspec} of contact three-manifolds, which we denote by $c_k$.
We now review and develop the facts about these invariants that we will need.

\subsection{Alternative ECH capacities}
\label{sec:altech}

Let $(X,\omega)$ be a symplectic four-manifold. The ``alternative ECH capacities'' defined in \cite{altech}, inspired by \cite{ms}, are a sequence of real numbers (or infinity) which we denote here by
\[
0 = c_0^{\op{Alt}}(X,\omega) < c_1^{\op{Alt}}(X,\omega) \le c_2^{\op{Alt}}(X,\omega) \le \cdots \le +\infty. 
\]
To review the definition, suppose first that $(X,\omega)$ is ``admissible'', meaning that each component of $(X,\omega)$ is either closed, or a compact Liouville domain such that the contact form on the boundary is nondegenerate. Then
\begin{equation}
\label{eqn:altech1}
c_k^{\op{Alt}}(X,\omega) = \sup_{\substack{J\in\mathcal{J}(\overline{X}) \\ x_1,\ldots,x_k\in X\mbox{\scriptsize distinct}}} \inf_{u\in\mathcal{M}^J(\overline{X};x_1,\ldots,x_k)}\mathcal{E}(u)\in[0,\infty].
\end{equation}
Here $\overline{X}$ denotes the ``symplectic completion'' of $X$ defined by attaching symplectization ends to each Liouville component, and $\mathcal{J}(\overline{X})$ denotes the set of ``cobordism-compatible'' almost complex structures on $\overline{X}$. Also $\mathcal{M}^J(\overline{X};x_1,\ldots,x_k)$ denotes the moduli space of possibly disconnected $J$-holomorphic curves in $\overline{X}$, nonconstant on every component, with punctures asymptotic to Reeb orbits in the symplectization ends, passing through the points $x_1,\ldots,x_k$. Finally $\mathcal{E}(u)$ denotes the sum over the closed components of $u$ of the symplectic area, plus the sum over the punctures of the period of the corresponding Reeb orbit. For a general symplectic four-manifold $(X',\omega')$, we define
\begin{equation}
\label{eqn:altech2}
c_k^{\op{Alt}}(X',\omega') = \sup\big\{c_k^{\op{Alt}}(X,\omega)\big\},
\end{equation}
where the supremum is over admissible $(X,\omega)$ for which there exists a symplectic embedding $(X,\omega)\hookrightarrow (X',\omega')$. See \cite{altech} for details.

We will need the following properties of the capacities $c_k^{\op{Alt}}$ proved in \cite[Thm.\ 6]{altech}. See \cite{altech} for more properties and computations.

\begin{lemma}
\label{lem:altechproperties}
The alternative ECH capacities $c_k^{\op{Alt}}$ have the following properties:
\begin{itemize}
\item (Monotonicity) 
If there exists a symplectic embedding $(X,\omega)\hookrightarrow (X',\omega')$, then 
\begin{equation}
\label{eqn:altechmonotonicity}
c_k^{\op{Alt}}(X,\omega) \le c_k^{\op{Alt}}(X',\omega')
\end{equation}
for all $k$.
\item (Disjoint Union)
\[
c_k^{\op{Alt}}\left(\coprod_{i=1}^m(X_i,\omega_i)\right) = \max_{k_1+\cdots+k_m=k}\sum_{i=1}^m c_{k_i}^{\op{Alt}}(X_i,\omega_i).
\]
\item{(Weyl law)}
If $U\subset\R^4$ is a bounded open set with piecewise smooth boundary, and if $\omega_0$ denotes the standard symplectic form on $\R^4$, then
\begin{equation}
\label{eqn:wlr}
c_k^{\op{Alt}}(U,\omega_0) = 2\op{vol}(U)^{1/2}k^{1/2} + O(k^{1/4}).
\end{equation}
\end{itemize}
\end{lemma}

\begin{remark}
\label{rem:altechsubleading}
The statement of the Weyl law \eqref{eqn:wlr} in \cite{altech} assumes smooth boundary, but the proof works for piecewise smooth boundary also. See \cite{cgh} for a version of this for domains with rough boundary. Edtmair \cite{oliverip} showed that in the case of smooth boundary, the $O(k^{1/4})$ error in \eqref{eqn:wlr} can be improved to $O(1)$. It is further conjectured in \cite[Conj.\ 1.5]{ruelle} that for the usual ECH capacities of a generic star-shaped domain with smooth boundary, the $O(1)$ term limits to $-1/2$ times the Ruelle invariant.
\end{remark}

We now prove a version of the Weyl law in a symplectization. If $Y$ is a closed three-manifold and $\lambda$ is a contact form on $Y$, recall that the {\bf symplectization\/} of $(Y,\lambda)$ is the symplectic four-manifold
\[
(\R\times Y, d(e^s\lambda))
\]
where $s$ denotes the $\R$ coordinate. If $f:Y\to[0,\infty)$ is a smooth function, define
\begin{equation}
\label{eqn:mf}
M_f = \{(s,y)\in\R\times Y \mid 0 < s < f(y)\}
\end{equation}
with the restriction of the symplectic form $d(e^s\lambda)$. Our convention is that the volume of a symplectic four-manifold $(M^4,\omega)$ is given by $\frac{1}{2}\int_M\omega\wedge\omega$, which for domains in $\R^4$ with the restriction of the standard symplectic form agrees with Euclidean volume. With this convention, in the symplectization we have
\begin{equation}
\label{eqn:volmf}
\op{vol}(M_f) = \frac{1}{2}\int_{M_f}d(e^s\lambda)\wedge d(e^s\lambda) = \frac{1}{2}\int_Y\left(e^{2f}-1\right)\lambda\wedge d\lambda.
\end{equation}

\begin{lemma}
\label{lem:sss}
Let $Y$ be a closed three-manifold, let $\lambda$ be a contact form on $Y$, and let $f:Y\to[0,\infty)$ be a smooth function. Then
\[
\liminf_{k\to\infty}\frac{c_k^{\op{Alt}}(M_f)^2}{k} \ge 4\op{vol}(M_f).
\]
\end{lemma}

\begin{proof}
Consider the standard Liouville form $\lambda_0$ on $\R^4$ defined by
\eqref{eqn:standardLiouvilleform}. Recall that $\lambda_0$ restricts to a contact form on $S^3$. Using Darboux's theorem for contact forms \cite[Thm.\ 2.5.1]{geiges}, we can find finitely many disjoint open sets $V_1,\ldots,V_m\subset Y$ such that:
\begin{itemize}
\item
$\partial V_i$ is piecewise smooth.
\item
$Y\setminus\coprod_{i=1}^m V_i$ has measure zero.
\item
There is an open set $U^i\subset S^3$ and a strict contactomorphism
\[
\phi_i: (U^i,{\lambda_0}|_{U^i})  \stackrel{\simeq}{\longrightarrow} (V_i,\lambda|_{V_i}).
\]
\end{itemize}
Write $f_i=f\circ\phi_i:U^i\to[0,\infty)$, and define
\[
U^i_{f_i} = \{(s,y)\in\R\times U^i \mid 0 < s < f_i(y)\}
\]
with the restriction of the symplectization symplectic form. Note that $U^i_{f_i}$ is symplectomorphic to an open set in $\R^4$. By the Weyl law \eqref{eqn:wlr}, there is a constant $C'$ such that for all $i\in\{1,\ldots,m\}$ and nonnegative integers $k_i$ we have
\begin{equation}
\label{eqn:usewlr}
c_{k_i}^{\op{Alt}}(U^i_{f_i}) \ge 2\op{vol}(U^i_{f_i})^{1/2}{k_i}^{1/2} - C'{k_i}^{1/4}.
\end{equation}

Fix a nonnegative integer $k$. Since
\[
\op{vol}(M_{ f}) = \sum_{i=1}^m\op{vol}(U^i_{f_i}),
\]
we can choose nonnegative integers $k_1,\ldots,k_m$ with
\begin{equation}
\label{eqn:sumki}
\sum_{i=1}^m k_i=k
\end{equation}
and
\begin{equation}
\label{eqn:chooseki}
k_i = \frac{\op{vol}(U^i_{ f_i})}{\op{vol}(M_{f})}k + O(1).
\end{equation}
Then by Lemma~\ref{lem:altechproperties} and equations \eqref{eqn:usewlr}--\eqref{eqn:chooseki}, we have
\[
\begin{split}
c_k^{\op{Alt}}(M_{f}) &\ge \sum_{i=1}^m c_{k_i}^{\op{Alt}}(U^i_{f_i})
\\
& \ge
\sum_{i=1}^m\left(2\op{vol}(U^i_{f_i})^{1/2}{k_i}^{1/2} - C'{k_i}^{1/4}\right)\\
&= 
2\op{vol}(M_{f})^{1/2}k^{1/2} + O(k^{1/4}).
\end{split}
\]
\end{proof}

\begin{remark}
Although we do not need this, the reverse inequality
\[
\limsup_{k\to\infty}\frac{c_k^{\op{Alt}}(M_f)^2}{k} \le 4\op{vol}(M_f)
\]
also holds, by the last two properties of the elementary spectral invariants in Proposition~\ref{prop:altspecproperties} below.
\end{remark}


\subsection{Elementary spectral invariants}
\label{sec:altspec}

Let $Y$ be a closed three-manifold and let $\lambda$ be a contact form on $Y$. The ``elementary spectral invariants'' defined in \cite{altspec} are a sequence of real numbers which we denote here by
\[
0 = c_0(Y,\lambda) < c_1(Y,\lambda) \le c_2(Y,\lambda) \le \cdots < \infty.
\]
These numbers are defined by a variant of \eqref{eqn:altech1} and \eqref{eqn:altech2} as follows. Suppose first that $\lambda$ is nondegenerate. Then
\[
c_k(Y,\lambda) = \sup_{R>0}\sup_{\substack{J\in\mathcal{J}_R(\R\times Y) \\ x_1,\ldots,x_k\in [-R,0]\times Y\mbox{\scriptsize distinct}}} \inf_{u\in\mathcal{M}^J(\R\times Y;x_1,\ldots,x_k)}\mathcal{E}_+(u)\in[0,\infty].
\]
Here $\mathcal{J}_R(\R\times Y)$ denotes the set of almost complex structures on $\R\times Y$ that are ``cobordism compatible'' where $\R\times Y$ is regarded as the symplectic completion of the cobordism $[-R,0]\times Y$ between $(Y,e^{-R}\lambda)$ and $(Y,\lambda)$. Also $\mathcal{M}^J(\R\times Y;x_1,\ldots,x_k)$ denotes the moduli space of possibly disconnected $J$-holomorphic curves in $\R\times Y$, nonconstant on every component, with positive punctures asymptotic to Reeb orbits as the $\R$ coordinate goes to $+\infty$, and negative punctures asymptotic to Reeb orbits as the $\R$ coordinate goes to $-\infty$, passing through the points $x_1,\ldots,x_k$. Finally $\mathcal{E}_+(u)$ denotes the sum over the positive punctures of the period of the corresponding Reeb orbit.
For a general contact form $\lambda$, we define
\[
c_k(Y,\lambda) = \sup_{f:Y\to \R_{<0}}c_k(Y,e^f\lambda) = \inf_{f:Y\to\R_{>0}}c_k(Y,e^f\lambda),
\]
where in the supremum and infimum we require that $e^f\lambda$ is nondegenerate. See \cite{altspec} for details.

We will need the following properties of the elementary spectral invariants proved in \cite{altspec}. To state them, define an {\bf orbit set\/} to be a finite set of pairs $\alpha = \{(\alpha_i,m_i)\}$ where the $\alpha_i$ are distinct simple Reeb orbits, and the $m_i$ are positive integers. Define the {\bf symplectic action\/}
\[
\mathcal{A}(\alpha) = \sum_im_i\mathcal{A}(\alpha_i)\in [0,\infty).
\]
Also, generalizing \eqref{eqn:mf}, if $f_1,f_2:Y\to\R$ are smooth functions with $f_1 \le f_2$, define
\[
M_{f_1,f_2} = \{(s,y)\in\R\times Y \mid f_1(y) < s < f_2(y)\}
\]
with the restriction of the symplectization symplectic form.

\begin{proposition}
\label{prop:altspecproperties}
Let $Y$ be a closed three-manifold and let $\lambda$ be a contact form on $Y$. The elementary spectral invariants $c_k(Y,\lambda)$ have the following properties:
\begin{description}
\item{(Conformality)}
If $r>0$ then $c_k(Y,r\lambda) = rc_k(Y,\lambda)$.
\item {($C^0$-Continuity)} For a fixed nonnegative integer $k$, the map $C^\infty(Y;\R)\to\R$ sending $f\mapsto c_k(Y,e^f\lambda)$ is $C^0$-continuous. 
\item {(Spectrality)}
For each nonnegative integer $k$, there exists an orbit set $\alpha$ such that $c_k(Y,\lambda) = \mathcal{A}(\alpha)$.
\item{(Weyl law)}
We have
\begin{equation}
\label{eqn:altspecweyl}
\lim_{k\to\infty}\frac{c_k(Y,\lambda)^2}{k} = 2\op{vol}(Y,\lambda).
\end{equation}
\item {(Capacity Bound)}
Let $f_1,f_2:Y\to\R$ be smooth functions with $f_1\le f_2$. Then for any nonnegative integers $k$ and $l$ we have
\begin{equation}
\label{eqn:capacitybound}
c_{k+l}(Y,e^{f_2}\lambda) \ge c_k(Y,e^{f_1}\lambda) + c_{l}^{\op{Alt}}(M_{f_1,f_2}).
\end{equation}
\end{description}
\end{proposition}

\begin{remark}
\label{rem:altspecmonotone}
As a special case of the Capacity Bound property, we have the following Monotonicity property: If $f:Y\to[0,\infty)$, then
\[
c_k(Y,\lambda) \le c_k(Y,e^f\lambda).
\]
\end{remark}

\begin{remark}
\label{rem:altspecsubleading}
The Weyl law can be refined to
\[
O(k^{1/4}) \le c_k(Y,\lambda) - \sqrt{2}(\op{vol}(Y,\lambda))^{1/2}k^{1/2} \le O(k^{2/5}).
\]
Here the left inequality follows similarly to the proof of Lemma~\ref{lem:sss}. The right inequality follows from the comparison with ECH spectral invariants in \cite[Thm.\ 6.1]{altspec}, together with the subleading asymptotics of the latter proved in \cite{subleading}.
\end{remark}

\begin{proof}[Proof of Proposition~\ref{prop:altspecproperties}.]
The Conformality, $C^0$-Continuity, and Spectrality properties are among the properties proved in \cite[Thm.\ 1.14]{altspec}. The Weyl law is proved in \cite[Thm.\ 1.19]{altspec}.

To prove the Capacity Bound property, suppose first that $f_1 < f_2$ and the contact forms $e^{f_i}\lambda$ are nondegenerate. In this case, it is shown in \cite[Lem.\ 4.4]{altspec} that if $(X,\omega)$ is a four-dimensional compact Liouville domain, and if there exists a symplectic embedding $(X,\omega)\hookrightarrow M_{f_1,f_2}$, then
\[
c_{k+l}(Y,e^{f_2}\lambda) \ge c_k(Y,e^{f_1}\lambda) + c_l^{\op{Alt}}(X,\omega).
\]
It then follows from the definition of $c_l^{\op{Alt}}(M_{f_1,f_2})$ that \eqref{eqn:capacitybound} holds in this case. The general case where $f_1\le f_2$ now follows from the $C^0$-continuity of $c_k$ and $c_{k+l}$ and the Monotonicity property of $c_l^{\op{Alt}}$.
\end{proof}

We will need the following refinement of the Spectrality property for one-parameter families of contact forms.

\begin{lemma}
\label{lem:derivative}
(cf. \cite[Lem.\ 3.4]{irieequi})
Let $Y$ be a closed three-manifold and let $(\lambda_\tau)_{\tau\in(-\epsilon,\epsilon)}$ be a smooth one-parameter family of contact forms on $Y$. Let $k$ be a nonnegative integer. Suppose that:
\begin{description}
\item{(i)} Every Reeb orbit for $\lambda_0$ with action $\le c_k(Y,\lambda_0)$ is nondegenerate.
\item{(ii)} The function $(-\epsilon,\epsilon)\to\R$ sending $\tau\mapsto c_k(Y,\lambda_\tau)$ is differentiable at $\tau=0$.
\end{description}
Then there exists an orbit set $\alpha$ for $\lambda_0$ such that
\[
c_k(Y,\lambda_0) = \int_\alpha\lambda_0
\]
and
\[
\frac{d}{d\tau}\bigg|_{\tau=0}c_k(Y,\lambda_\tau) = \int_\alpha\frac{d}{d\tau}\bigg|_{\tau=0}\lambda_\tau.
\]
\end{lemma} 

\begin{proof}
By the hypothesis (i), there are only finitely many orbit sets for $\lambda_0$ with action $\le c_k(Y,\lambda_0)$. Denote these orbit sets by $\alpha_0^1,\ldots,\alpha_0^m$. It also follows from hypothesis (i) that if $|\tau|>0$ is small, then these orbit sets persist to orbit sets $\alpha_\tau^1,\ldots,\alpha_\tau^m$ for $\lambda_\tau$, and these include all of the orbit sets for $\lambda_\tau$ with action $\le c_k(Y,\lambda_{\tau})$. Thus by the Spectrality property in Proposition~\ref{prop:altspecproperties}, for each $\tau$ with $|\tau|$ sufficiently small, there exists $i\in\{1,\ldots,m\}$ such that
\begin{equation}
\label{eqn:lambdatau}
c_k(Y,\lambda_\tau) = \int_{\alpha_\tau^i}\lambda_\tau.
\end{equation}
Now let $(\tau_j)_{j=1,\ldots}$ be a sequence of positive numbers converging to zero. Then there must exist a single index $i\in\{1,\ldots,m\}$ which fulfills \eqref{eqn:lambdatau} for infinitely many numbers $\tau_j$ in the sequence. So after passing to a subsequence, \eqref{eqn:lambdatau} holds for all $\tau_j$ for this fixed index $i$. By the $C^0$-Continuity property in Proposition~\ref{prop:altspecproperties}, we have
\[
c_k(Y,\lambda_0) = \int_{\alpha_0^i}\lambda_0.
\]
Given hypothesis (ii), we also have
\[
\begin{split}
\frac{d}{d\tau}\bigg|_{\tau=0}c_k(Y,\lambda_\tau) &= \lim_{j\to\infty}\frac{\int_{\alpha_{\tau_j}^i}\lambda_{\tau_j} - \int_{\alpha_0^i}\lambda_0}{\tau_j}\\
&= \frac{d}{d\tau}\bigg|_{\tau = 0}\int_{\alpha_\tau^i}\lambda_\tau\\
&= \int_{\alpha_0^i}\frac{d}{d\tau}\bigg|_{\tau=0}\lambda_\tau.
\end{split}
\]
Here in the last line there is no term involving the derivative of $\alpha_\tau^i$ because $\alpha_0^i$ is an orbit set for $\lambda_0$.
\end{proof}

\section{Proof of the main theorem}
\label{sec:proof}

We now prove Theorem~\ref{thm:main1}.

Throughout this section, if $r_0>0$, let $\beta:[0,r_0)\to[0,1]$ denote a cutoff function which equals $0$ on the interval $[0,r_0/3]$, which equals $1$ on the interval $[2r_0/3,r_0)$, and whose derivative satisfies $\beta'(r)\in[0,4r_0^{-1}]$ (this last condition is needed in Remark~\ref{rem:deltass} below).


\subsection{Reduction to the very nice case}
\label{sec:rnc}

Let $u:\Sigma\to Y$ be an admissible symplectic surface. We now consider a situation in which the contact form and the surface have a simple form near the boundary. Roughly speaking, the condition on the surface is that for each boundary Reeb orbit, the angles of the conormal vectors of the surface are evenly spaced in the normal bundle to the Reeb orbit, and rotate at uniform speed as one moves along the Reeb orbit. To be precise:

\begin{definition}
\label{def:verynice}
Suppose that $(Y,\lambda,u)$ is nice as in Definition~\ref{def:newnice}. We say that $(Y,\lambda,u)$ is {\bf very nice\/} if for each simple Reeb orbit $\gamma:\R/T\Z\to Y$ in $u(\partial\Sigma)$, the neighborhood identification \eqref{eqn:Ngammanew} in Definition~\ref{def:newnice} can be chosen with the following properties:
\begin{description}
\item{(a)}
In the notation of \eqref{eqn:modelReeb} we have\footnote{Note that \eqref{eqn:modelcontact} implies \eqref{eqn:modelReeb}.}
\begin{equation}
\label{eqn:modelcontact}
\psi^*\lambda = dt + \frac{r^2}{2}d\left(\theta - \frac{2\pi\rho}{T}t\right).
\end{equation}
\item{(b)}
Let $q$ denote the signed covering multiplicity of each component of $\partial\Sigma$ mapping to $\gamma$. Then: 
\begin{itemize}
\item
There exists an integer $p$ relatively prime to $q$ such that if $B$ is a component of $\partial\Sigma$ with $u(B)=\gamma(\R/T\Z)$, and if $N(B)$ denotes the connected component of $u^{-1}(N(\gamma))$ containing $B$, then there exists $\theta_B\in \R/\Z$ such that
\begin{equation}
\label{eqn:nicesigma}
\psi^{-1}(u(N(B))) = \left\{(t,r,\theta) \;\bigg|\; \theta = 2\pi\left(\theta_B + \frac{j}{|q|} +  \frac{pt}{|q|T}\right), \; j=0,\ldots,|q| - 1\right\}.
\end{equation}
Here $t$ denotes the $\R/T\Z$ coordinate, and $r,\theta$ are polar coordinates on $D^2(r_0)$.
\item
If $\partial\Sigma$ has $m$ components mapping to $\gamma$, then the numbers $\theta_B$ for the different components differ by integer multiples of $1/(mq)$.
\end{itemize}
\end{description}
\end{definition}

\begin{lemma}
\label{lem:isotopesurface}
Let $\lambda$ be a contact form on $Y$, and let $u:\Sigma\to Y$ be a nice admissible symplectic surface. Then there is an admissible symplectic surface $u':\Sigma'\to Y$ such that:
\begin{itemize}
\item
$u'$ is isotopic to $u$ rel boundary.
\item
$(Y,\lambda,u')$ is very nice.
\end{itemize}
\end{lemma}

\begin{remark}
\label{rem:kindanice}
Since $u'$ is isotopic to $u$ rel boundary, both sides of the inequality \eqref{eqn:boxed} in Theorem~\ref{thm:main1} are the same for $u$ and $u'$. Thus Lemma~\ref{lem:isotopesurface} implies that in proving Theorem~\ref{thm:main1}, we can assume without loss of generality that the given admissible symplectic surface is very nice.
\end{remark}

\begin{proof}[Proof of Lemma~\ref{lem:isotopesurface}]
By shrinking $r_0$ if necessary, we can assume that the neighborhoods $N(\gamma)$ for different simple Reeb orbits $\gamma$ in $u(\partial\Sigma)$ are disjoint. We now proceed in two steps.

{\em Step 1.\/} Fix a simple Reeb orbit $\gamma$ in $u(\partial\Sigma)$ and a neighborhood identification \eqref{eqn:Ngammanew}. We now show that, after possibly shrinking the neighborhood $N(\gamma)$, the identification \eqref{eqn:Ngammanew} can be modified by a change of coordinates to satisfy condition (a) in Definition~\ref{def:verynice}.

On $(\R/T\Z)\times D^2(r_0)$, let $\lambda_0$ denote the contact form given by the right hand side of \eqref{eqn:modelcontact}, and let $\lambda_1=\psi^*\lambda$. Observe that $\lambda_0$ and $\lambda_1$ have the same Reeb vector field. For $\tau\in[0,1]$, consider the contact form $\lambda_\tau=(1-\tau)\lambda_0+\tau\lambda_1$. We claim that there exist $r_0'\in(0,r_0)$ and an isotopy
\[
\left\{\phi_\tau:(\R/T\Z)\times D^2(r_0')\longrightarrow (\R/T\Z)\times D^2(r_0)\right\}_{\tau\in[0,1]}
\]
such that (i) $\phi_0$ is the identity map, (ii) $\phi_\tau^*\lambda_\tau=\lambda_0$, and (iii) $\phi_\tau$ fixes $(\R/T\Z)\times\{0\}$. Assuming this claim, we then achieve condition (a) in Definition~\ref{def:verynice} by replacing $(r_0,\psi)$ by $(r_0',\psi\circ\phi_1)$.

To prove the claim, we use a Moser type argument similar to \cite[Prop.\ 2.1]{geigesvolume}. Let $X_\tau$ denote the derivative of the isotopy we are seeking. We need $X_\tau$ to satisfy
\begin{equation}
\label{eqn:Moser}
d\imath_{X_\tau}\lambda_\tau+\imath_{X_\tau}d\lambda_\tau = \lambda_0 - \lambda_1, \quad\quad\quad X_\tau|_{(\R/T\Z)\times\{0\}}=0.
\end{equation}
One way to solve the first equation in \eqref{eqn:Moser} is to require
\[
\imath_{X_\tau}\lambda_\tau = 0, \quad\quad\quad \imath_{X_\tau}d\lambda_\tau = \lambda_0 - \lambda_1.
\]
These two equations have a unique solution because the right hand side of the second equation annihilates the Reeb vector field. If the rotation number $\rho$ is not an integer, then since $\lambda_0-\lambda_1$ is preserved by the Reeb flow, it is zero along $(\R/T\Z)\times\{0\}$, so $X_\tau$ will also satisfy the second equation in \eqref{eqn:Moser}.

If the rotation number is an integer, then we use a different method to solve \eqref{eqn:Moser}. By a change of trivialization we can assume that $\rho=0$, so that the Reeb vector field is $\partial_t$. Along $(\R/T\Z)\times\{0\}$ we can write
\[
\lambda_1-\lambda_0 = a\,dx + b\,dy,
\]
where $a$ and $b$ are constants because $\lambda_0$ and $\lambda_1$ are both invariant under the flow of the Reeb vector field $\partial_t$. There is now a unique smooth vector field $X_\tau$ such that
\[
\imath_{X_\tau}\lambda_\tau = -ax-by, \quad\quad\quad  \imath_{X_\tau}d\lambda_\tau = \lambda_0 - \lambda_1 + a\,dx + b\,dy,
\]
because the right hand side of the second equation annihilates the Reeb vector field. The above two equations imply the conditions \eqref{eqn:Moser}, and this completes the proof of the claim.

{\em Step 2.\/} We now construct an admissible symplectic surface $u':\Sigma'\to Y$, isotopic to $u$ rel boundary, such that $(Y,\lambda,u')$ is very nice using the neighborhood identifications chosen in Step 1 (after possibly shrinking the neighborhoods further).

Let $B$ be a boundary component of $\Sigma$ mapping to $\gamma$. 
Suppose to start that $B$ is the only boundary component of $\Sigma$ mapping to $\gamma$, and that the boundary orientation of $\partial\Sigma$ agrees with the Reeb orientation of $\gamma$.

By decreasing $r_0$ if necessary, we can assume that $\psi^{-1}(u(N(B)))$ has a parametrization of the form
\begin{equation}
\label{eqn:Nparametrization}
\begin{split}
(\R/qT\Z) \times [0,r_0) & \longrightarrow (\R/T\Z) \times D^2(r_0),\\
(\tilde{t},r) & \longmapsto \left(t, r, 2\pi \eta\left(\tilde{t},r\right)\right)
\end{split}
\end{equation}
where $\eta: (\R/qT\Z) \times [0,r_0) \longrightarrow \R/\Z$, and $t$ denotes the projection of $\tilde{t}$ to $\R/T\Z$.

Let $p\in\Z$ denote the degree of the map $\R/qT\Z\to \R/\Z$ sending $\tilde{t}\mapsto \eta(\tilde{t},r)$ for fixed $r>0$. Since $u|_{\op{int}(\Sigma)}$ is an embedding, $p$ is relatively prime to $q$.

Since $\op{int}(\Sigma)$ is positively transverse to the Reeb vector field, it follows from \eqref{eqn:Nparametrization} and our hypothesis on the boundary orientation that if $r>0$ is small, then
\begin{equation}
\label{eqn:partialeta}
\frac{\partial\eta}{\partial \tilde{t}} < \frac{\op{rot}_\tau(\gamma)}{T}.
\end{equation}
Here $\tau$ denotes the trivialization of $\gamma^*\xi$ determined by the derivative of the neighborhood identification \eqref{eqn:Ngammanew}. Integrating \eqref{eqn:partialeta} over $\widetilde{t}$ for fixed small $r>0$ gives
\begin{equation}
\label{eqn:pqr}
\frac{p}{q} < \op{rot}_\tau(\gamma).
\end{equation}

Pick an arbitrary $\theta_B\in\R/\Z$. We now replace $u(N(B))$ with $u'(N(B))$ defined by replacing $\eta$ in \eqref{eqn:Nparametrization} with the function
\[
\eta': (\R/qT\Z) \times [0,r_0) \longrightarrow \R/\Z
\]
defined by
\[
\eta'\left(\tilde{t},r\right) = (1-\beta(r)) \left(\widetilde{\theta}_B + \frac{p\tilde{t}}{qT}\right) + \beta(r)\widetilde{\eta}\left(\tilde{t},r\right).
\]
Here $\widetilde{\theta}_B$ is a real lift of $\theta_B$, and $\widetilde{\eta}:[0,qT)\times[0,r_0)\to\R$ is a real-valued lift of $\eta$. Then $u'(N(B))$ satisfies \eqref{eqn:nicesigma} for a smaller neighborhood $N(\gamma)$. Furthermore, by \eqref{eqn:partialeta} and \eqref{eqn:pqr}, if $r_0$ is sufficiently small then
\begin{equation}
\label{eqn:stilltransverse}
\frac{\partial \eta'}{\partial \tilde{t}} < \frac{\op{rot}_\tau(\gamma)}{T}.
\end{equation}
Thus if $r_0$ is sufficiently small, then $u'(N(B))$ is still positively transverse to the Reeb vector field $R_\lambda$.

If the boundary orientation of $B$ is opposite the Reeb orientation of $\gamma$, then the construction above works the same, except that $q$ needs to replaced by $|q|$, and the directions of the inequalities in \eqref{eqn:partialeta}, \eqref{eqn:pqr}, and \eqref{eqn:stilltransverse} are switched. If $\partial\Sigma$ has $m>1$ boundary components $B$ mapping to $\gamma$, then we define the modifications $u'(N(B))$ as above, using different values of $\theta_B$ that differ by integer multiples of $1/(mq)$, cyclically ordered in accordance with the incoming ends of $\Sigma$, and the resulting surface $\Sigma'_\gamma$ will still be an embedding on its interior.

The surface $u':\Sigma'\to Y$ is now defined by modifying $u(\Sigma)$ in a neighborhood of each Reeb orbit $\gamma$ in $u(\partial\Sigma)$ as above, and keeping everything the same outside of these neighborhoods.
\end{proof}


\subsection{Inflating the contact form in a slab}
\label{sec:inflation}

We now discuss a certain deformation of the contact form which will be needed for the proof of Theorem~\ref{thm:main1}.

We begin by specifying various choices that will enter into the construction. Suppose that $(Y,\lambda,u)$ is very nice as in Definition~\ref{def:verynice}. Fix $r_0>0$ sufficiently small so that $r_0$ can be used for all of the neighborhood identifications \eqref{eqn:Ngammanew} in Definition~\ref{def:verynice}. Write
\[
\Sigma_0 = \Sigma \setminus \bigcup_{\gamma\subset u(\partial\Sigma)} u^{-1}(N(\gamma)).
\]
Let $\epsilon_0>0$ be given.

By \eqref{eqn:nicesigma}, if $s_0>0$ is sufficiently small, then the Reeb flow
\begin{equation}
\label{eqn:s0}
\begin{split}
[0,s_0]\times\op{int}(\Sigma) & \longrightarrow Y,\\
(s,z) &\longmapsto \Phi(s,u(z))
\end{split}
\end{equation}
is an embedding. Choose such an $s_0$. Also assume that $s_0$ is sufficiently small so that for every simple Reeb orbit $\gamma$ in $u(\partial\Sigma)$, we have
\begin{equation}
\label{eqn:s0small}
\left(\frac{\mathcal{A}(\gamma)}{\gamma\cdot\Sigma}-s_0\right)^{-1} < \frac{\gamma\cdot\Sigma}{\mathcal{A}(\gamma)}+\epsilon_0.
\end{equation}
Fix a smooth function $\zeta:[0,s_0]\to[0,1/2]$, not identically zero, with $\zeta(s)=0$ for $s$ close to $0$ or $s_0$, with derivative $\zeta'(s)\ge 0$ for $s\in[0,s_0/2]$, and with
\begin{equation}
\label{eqn:zetasymmetry}
\zeta(s)=\zeta(s_0-s).
\end{equation}
Given $\delta\ge 0$, let $\bar{\delta}\ge 0$ denote the unique number such that
\begin{equation}
\label{eqn:uniquedelta}
\int_0^{s_0}e^{\bar{\delta}\zeta(s)}ds = s_0 + \delta.
\end{equation}

\begin{lemma}
\label{lem:slab}
Let $u:\Sigma\to(Y,\lambda)$ be an admissible symplectic surface. Assume that $(Y,\lambda,u)$ is very nice and that $\lambda$ is nondegenerate. Choose $r_0$, $\epsilon_0$, $s_0$, and $\zeta$ as above. Then there exist a smooth family of contact forms $\{\lambda_\delta\}_{\delta\ge0}$ with $\lambda_0=\lambda$, and for each $\delta>0$ sufficiently small an injection $\imath_\delta: \mathcal{P}(\lambda) \hookrightarrow \mathcal{P}(\lambda_\delta)$, with the following properties:
\begin{description}
\item{(a)}
We have $\lambda_\delta=e^{f_\delta}\lambda$ for a function $f_\delta:Y\to[0,\bar{\delta}]$ with $\partial_\delta f_\delta\ge 0$ such that if $s\in[0,s_0]$ and $z\in\Sigma_0$, then $f_\delta(\Phi(s,u(z)))=\bar{\delta}\zeta(s)$.
\item{(b)}
If $\gamma\in\mathcal{P}(\lambda)$, then:
\begin{description}
\item{(i)}
 $\imath_\delta(\gamma)$ is nondegenerate, and $\imath_\delta(\gamma)\cdot\Sigma=\gamma\cdot\Sigma$.
\item{(ii)}
We have
\[
\mathcal{A}_{\lambda_\delta}(\imath_\delta(\gamma)) = \left\{\begin{array}{cl} 
\mathcal{A}_\lambda(\gamma) + \delta(\gamma\cdot\Sigma),
&
\gamma\not\subset u(\partial\Sigma),\\
\mathcal{A}_\lambda(\gamma),
&
\gamma\subset u(\partial\Sigma).
\end{array}
\right.
\]
\end{description}
\item{(c)}
If $\gamma'\in\mathcal{P}(\lambda_\delta)\setminus\imath_\delta(\mathcal{P}(\lambda))$, then:
\begin{description}
\item{(i)}
There is a simple Reeb orbit $\gamma$ in $u(\partial\Sigma)$ such that
\begin{equation}
\label{eqn:mqrho}
\frac{\gamma'\cdot\Sigma}{\mathcal{A}_{\lambda_\delta}(\gamma')} < \frac{\gamma\cdot\Sigma}{\mathcal{A}_\lambda(\gamma)} + \epsilon_0.
\end{equation}
\item{(ii)} 
We have
\begin{equation}
\label{eqn:derivativesmall}
\int_{\gamma'}\frac{d}{d\delta}\lambda_\delta \le (1+\epsilon_0)(\gamma'\cdot\Sigma).
\end{equation}
\end{description}
\item{(d)}
For almost every $\delta$: The contact form $\lambda_\delta$ is nondegenerate, and for each nonnegative integer $k$ the function $\delta\mapsto c_k(Y,\lambda_\delta)$ is differentiable at $\delta$.
\end{description}
\end{lemma}

\begin{proof}
We proceed in seven steps.

{\em Step 1.\/} We first give a provisional definition of $\{\lambda_\delta\}$, which we will show satisfies properties (a)--(c). In Step 7 below we will argue that the family $\{\lambda_\delta\}$ can be perturbed, maintaining properties (a)--(c), so that property (d) also holds.

Define a function $\beta_\Sigma:\Sigma\to[0,1]$ as follows. If $\gamma$ is a simple Reeb orbit in $u(\partial\Sigma)$, then for $z\in u^{-1}(N(\gamma))$, identified with a point $(t,r,\theta)\in (\R/T\Z)\times D^2(r_0)$, we define $\beta_\Sigma(z) = \beta(r)$, where the cutoff function $\beta$ was chosen at the beginning of \S\ref{sec:proof}. On the rest of $\Sigma$ we define $\beta_\Sigma\equiv 1$.

Let $\mathcal{S}$ denote the image of the map \eqref{eqn:s0}, which we identify with $[0,s_0]\times\op{int}(\Sigma)$ with coordinates $s,z$. Define a contact form $\lambda_\delta$ on $Y$ by
\begin{equation}
\label{eqn:lambdadelta}
\begin{split}
\lambda_\delta|_\mathcal{S} &= e^{\bar{\delta}\zeta(s)\beta_\Sigma(z)}\lambda,\\
\lambda_\delta|_{(Y\setminus \mathcal{S})} &= \lambda.
\end{split}
\end{equation}
Then property (a) holds by construction.

{\em Step 2.\/} We now compute the Reeb vector field $R_{\lambda_\delta}$ on $\mathcal{S}$, in equations \eqref{eqn:ReebNB} and \eqref{eqn:ReebSigma0} below. To start, note that since the Reeb vector field $R_\lambda$ on $\mathcal{S}$ is given by $\partial_s$, we have
\[
\lambda|_{\mathcal{S}} = ds + \lambda_\Sigma,
\]
where $\lambda_\Sigma$ is a one-form on $\op{int}(\Sigma)$ which does not depend on $s\in[0,s_0]$. Thus
\begin{equation}
\label{eqn:lambdadeltaS}
{\lambda_\delta}|_{\mathcal{S}} = e^{\bar{\delta}\zeta(s)\beta_\Sigma(z)}(ds+\lambda_\Sigma).
\end{equation}

Next, if $B$ is a boundary component of $\Sigma$, let $\gamma\in\mathcal{P}(\lambda)$ denote the corresponding simple Reeb orbit, and let $N(B)$ denote the interior of the corresponding component of $u^{-1}(N(\gamma))\subset\Sigma$. By the very niceness hypothesis and Definition~\ref{def:verynice}, under the neighborhood identification \eqref{eqn:Ngammanew} we have
\begin{equation}
\label{eqn:nicelambda}
\psi^*\lambda = dt + \frac{r^2}{2} d\left(\theta - \frac{2\pi\op{rot}_\tau(\gamma)}{T}t\right).
\end{equation}
Let $p$ and $q$ be the integers associated to $\gamma$ in Definition~\ref{def:verynice}.  By \eqref{eqn:nicesigma}, we can parametrize $N(B)$ by $(\R/|q|T\Z)\times (0,r_0)$ via the map sending
\begin{equation}
\label{eqn:parametrization}
(\tilde{t},r) \longmapsto \left(t,r,2\pi\left(\theta_B + \frac{p\widetilde{t}}{|q|T}\right)\right)
\end{equation}
for some $\theta_B\in\R/\Z$, where $t$ denotes the projection of $\widetilde{t}\in\R/|q|T\Z$ to $\R/T\Z$. Write
\[
\rho = \op{rot}_\Sigma(\gamma) = \op{rot}_\tau(\gamma) - \frac{p}{|q|},
\]
where $\tau$ is as in \S\ref{sec:rnc}. By \eqref{eqn:nicelambda}, the restriction of $\lambda_\Sigma$ to $N(B)$ is given in the parametrization \eqref{eqn:parametrization} by
\[
{\lambda_\Sigma}|_{N(B)} = \left(1-\frac{\pi \rho r^2}{T}\right)d\widetilde{t}.
\]
Thus by \eqref{eqn:lambdadeltaS}, in $[0,s_0]\times N(B)$, using the parametrization \eqref{eqn:parametrization} of $N(B)$, we have
\begin{equation}
\label{eqn:snb}
{\lambda_\delta}\big|_{[0,s_0]\times N(B)} = e^{\bar{\delta}\zeta(s)\beta(r)}\left(ds + \left(1-\frac{\pi \rho r^2}{T}\right)d\widetilde{t}\right).
\end{equation}

A computation using equation \eqref{eqn:snb} shows that the Reeb vector field $R_{\lambda_\delta}$ on $[0,s_0]\times N(B)$ is given in the coordinates $(s,\tilde{t},r)$ by
\begin{equation}
\label{eqn:ReebNB}
\begin{split}
\frac{2\pi\rho r}{T}e^{\bar{\delta}\zeta(s)\beta(r)} {R_{\lambda_\delta}}\big|_{[0,s_0]\times N(B)} =&\;
\bar{\delta}\zeta'(s)\beta(r)\left(1-\frac{\pi\rho r^2}{T}\right)\partial_r\\
&+\bar{\delta}\zeta(s)\beta'(r)\partial_{\widetilde{t}}
\\
&+\left(\frac{2\pi\rho r}{T} - \bar{\delta}\zeta(s)\beta'(r)\left(1-\frac{\pi\rho r^2}{T}\right)\right)\partial_s.
\end{split}
\end{equation}

To compute the Reeb vector field $R_{\lambda_\delta}$ on the rest of $\mathcal{S}$, observe from \eqref{eqn:lambdadeltaS} that on $[0,s_0]\times\Sigma_0\subset\mathcal{S}$, we have
\begin{equation}
\label{eqn:lambdadeltasigma0}
{\lambda_\delta}|_{[0,s_0]\times\Sigma_0} = e^{\bar{\delta} \zeta(s)}\left(ds + \lambda_\Sigma\right).
\end{equation}
Since $d\lambda_\Sigma$ is a symplectic form on $\Sigma_0$, there is a unique (Liouville) vector field $X$ on $\Sigma_0$ such that 
\begin{equation}
\label{eqn:uniqueX}
\imath_Xd\lambda_\Sigma = \lambda_\Sigma.
\end{equation}
It then follows from \eqref{eqn:lambdadeltasigma0} and \eqref{eqn:uniqueX} that the Reeb vector field associated to $\lambda_\delta$ on $[0,s_0]\times\Sigma_0$ is given by
\begin{equation}
\label{eqn:ReebSigma0}
{R_{\lambda_\delta}}|_{[0,s_0]\times\Sigma_0} = e^{-\bar{\delta}\zeta(s)}\left(\partial_s - \bar{\delta}\zeta'(s)X\right).
\end{equation}

{\em Step 3.\/} We now read off from \eqref{eqn:ReebNB} and \eqref{eqn:ReebSigma0} some useful information about the Reeb flow on $\mathcal{S}$.

To start, observe that if $\delta$ is sufficiently small, then the coefficient of $\partial_s$ is everywhere positive in \eqref{eqn:ReebNB} and \eqref{eqn:ReebSigma0}. In particular, $R_{\lambda_\delta}$ has no periodic orbits contained in $\mathcal{S}$. Assume henceforth that $\delta$ is sufficiently small in this sense.

It follows from \eqref{eqn:ReebNB} that in $[0,s_0]$ cross the interior of a neighborhood of $\partial\Sigma$, namely where $r\in(0,r_0/3)$, the coefficient of $\partial_r$ is zero. We conclude that a trajectory of the Reeb vector field $R_{\lambda_\delta}$ starting on $\{0\}\times\op{int}(\Sigma)$ will stay in $\mathcal{S}$ until it reaches $\{s_0\}\times\op{int}(\Sigma)$. We then have a well-defined diffeomorphism
\[
P_\delta : \op{int}(\Sigma) \stackrel{\simeq}{\longrightarrow} \op{int}(\Sigma),
\]
where given $z\in\op{int}(\Sigma)$, the trajectory of $R_{\lambda_\delta}$ starting at $(0,z)$ will flow in $\mathcal{S}$ to $(s_0,P_\delta(z))$.

Next, let $\Sigma_1$ denote the union of $\Sigma_0$ with the union over $B$ of the set of points in $N(B)$ for which $r\in[\frac{2}{3}r_0,r_0)$, so that $\beta(r)=1$. If $\delta$ is sufficiently small, so that the absolute value of the coefficient of $\partial_r$ in \eqref{eqn:ReebNB} is sufficiently small in $[0,s_0]\times(\Sigma_1\setminus\Sigma_0)$, then no trajectory of $R_{\lambda_\delta}$ in $\mathcal{S}$ can intersect both $[0,s_0]\times(\Sigma\setminus\Sigma_1)$ and $[0,s_0]\times\Sigma_0$. Assume henceforth that $\delta$ is sufficiently small in this sense.

By the previous paragraph, if $z\in\Sigma_0$, then the flow of $R_{\lambda_\delta}$ starting at $(0,z)$ will reach $\{s_0/2\}\times\op{int}(\Sigma)$ at a point in $\{s_0/2\}\times\Sigma_1$. By \eqref{eqn:zetasymmetry}, on $[0,s_0]\times\Sigma_1$, if we replace $s\leftrightarrow s_0-s$, then the coefficient of $\partial_s$ in equations \eqref{eqn:ReebNB} and \eqref{eqn:ReebSigma0} is unchanged, while the rest of the Reeb vector field is multiplied by $-1$. Hence the flow of $-R_{\lambda_\delta}$ starting at $(s_0,z)$ will reach the same point in $\{s_0/2\}\times\Sigma_1$. We conclude that
\begin{equation}
\label{eqn:Pdeltasigma1}
{P_\delta}\big|_{\Sigma_0} = \op{id}_{\Sigma_0}.
\end{equation}

{\em Step 4.\/}
We now define the injection $\imath_\delta$ and establish property (b).

If $\gamma\in\mathcal{P}(\lambda)$ and $\gamma$ does not intersect $\op{int}(\Sigma)$, then $\gamma$ is also a periodic orbit of $R_{\lambda_\delta}$ with the same action, which is nondegenerate by our hypothesis that $\lambda$ is nondegenerate, and we define $\imath_\delta(\gamma)=\gamma$.

If $\gamma\in\mathcal{P}(\lambda)$ and $\gamma$ does intersect $\op{int}(\Sigma)$, then $\gamma\cap\op{int}(\Sigma)\subset\Sigma_0$. This is because if $\gamma'$ is a simple Reeb orbit in $u(\partial\Sigma)$, then equation \eqref{eqn:modelReeb} implies that the flow of $R_\lambda$ in $N(\gamma')$ stays in $N(\gamma')$, and there are no simple periodic orbits there other than $\gamma'$, since $\gamma'$ and its iterates are nondegenerate by hypothesis. Since $\gamma\cap\Sigma\subset\Sigma_0$, it follows from \eqref{eqn:Pdeltasigma1} that there is a unique Reeb orbit in $\mathcal{P}(\lambda_\delta)$ which agrees with $\gamma$ in $Y\setminus\mathcal{S}$, and we define $\imath_\delta(\gamma)$ to be this orbit. Equation \eqref{eqn:Pdeltasigma1} also implies that $\gamma$ and $\imath_\delta(\gamma)$ have the same linearized return map, so since $\gamma$ is nondegenerate by hypothesis, $\imath_\delta(\gamma)$ is also nondegenerate.

We have defined the injection $\imath_\delta$, and property (b)(i) holds by construction. Property (b)(ii) holds because by Step 3, for $z\in\Sigma_0$, the trajectory of $R_{\lambda_\delta}$ in $\mathcal{S}$ from $(0,z)$ to $(s_0,z)$ stays within $[0,s_0]\times\Sigma_1$. By equations \eqref{eqn:uniquedelta}, \eqref{eqn:ReebNB}, and \eqref{eqn:lambdadeltasigma0}, the flow time of this trajectory is $s_0+\delta$.

{\em Step 5.\/} We now prove property (c)(i).

Let $\gamma'\in\mathcal{P}(\lambda_\delta) \setminus \imath_\delta(\mathcal{P}(\lambda))$.
Write $k=\gamma'\cdot\Sigma$. Then by Step 3, there are $k$ disjoint subtrajectories of $\gamma'$ in $\mathcal{S}$, each of which starts on $\{0\}\times\Sigma$ and ends on $\{s_0\}\times\Sigma$. Order these subtrajectories arbitrarily and denote them by $\gamma'_1,\ldots,\gamma'_k$.

Since $\gamma'\notin \imath_\delta(\mathcal{P}(\lambda))$, it follows that for each $i\in\{1,\ldots,k\}$, there is a component $B(i)$ of $\partial\Sigma$ such that the subtrajectory $\gamma'_i$ starts in $\{0\}\times N(B(i))$. The subtrajectory must end in $\{s_0\}\times N(B(i))$, because by \eqref{eqn:Pdeltasigma1}, the map $P_\delta$ sends each component of $\Sigma\setminus\Sigma_0$ to itself. It then follows from \eqref{eqn:modelReeb} that there is a simple Reeb orbit $\gamma$ in $u(\partial\Sigma)$ such that all of the subtrajectories $\gamma_i'$ start and end in $N(\gamma)$, and the rest of the Reeb orbit $\gamma'$ is contained within $N(\gamma)\setminus\mathcal{S}$.

Let $p$ and $q$ be the integers associated to $\gamma$ in Definition~\ref{def:verynice}. Let $m$ denote the number of components of $\partial\Sigma$ mapping to $\gamma$. By equations \eqref{eqn:modelReeb} and \eqref{eqn:nicesigma} and Definition~\ref{def:gammadotsigma}, the Reeb flow of $\lambda$ takes time $\mathcal{A}_\lambda(\gamma)/(\gamma\cdot\Sigma)$ to flow within $N(\gamma)$ from $\Sigma$ to itself.
Thus each subtrajectory of $\gamma'$ in between subtrajectories $\gamma_i'$ has flow time equal to $\mathcal{A}_\lambda(\gamma)/(\gamma\cdot\Sigma) - s_0$.
Since the number of such subtrajectories equals $\gamma'\cdot\Sigma$, the total flow time of $\gamma'$ satisfies
\[
\mathcal{A}_{\lambda_\delta}(\gamma') > (\gamma'\cdot\Sigma)\left(\frac{\mathcal{A}_\lambda(\gamma)}{\gamma\cdot\Sigma} - s_0\right).
\]
Together with our assumption \eqref{eqn:s0small}, this implies \eqref{eqn:mqrho}. This proves property (c)(i).

{\em Step 6.\/}
We now prove property (c)(ii).

Using the notation $\gamma_i'$ from Step 5, by equation \eqref{eqn:lambdadelta} we have
\[
\int_{\gamma'}\frac{d}{d\delta}\lambda_\delta = \frac{d\bar{\delta}}{d\delta}\sum_{i=1}^k\int_{\gamma'_i}\zeta(s)\beta_\Sigma(z)\lambda_\delta.
\]
By changing variables, we have
\[
\int_{\gamma_i'}\zeta(s)\beta_\Sigma(z)\lambda_\delta = \int_{\gamma_i'}\zeta(s)\beta_\Sigma(z)(R_{\lambda_\delta}s)^{-1}ds,
\]
where $R_{\lambda_\delta}s$ is the coefficient of $\partial_s$ in \eqref{eqn:ReebNB} or \eqref{eqn:ReebSigma0} as appropriate. It follows from \eqref{eqn:ReebSigma0} that on $[0,s_0]\times\Sigma_0$ we have
\[
R_{\lambda_\delta}s = e^{-\bar{\delta}\zeta(s)}.
\]
It follows from \eqref{eqn:ReebNB} that if $\delta$ is sufficiently small, then on $[0,s_0]\times(\Sigma\setminus\Sigma_0)$ we have
\[
R_{\lambda_\delta}s \ge (1+\epsilon_0)^{-1}e^{-\bar{\delta}\zeta(s)\beta_\Sigma(z)}.
\]
Assume that $\delta$ is sufficiently small in this sense.
Combining the above four lines and using the fact that $\beta_\Sigma(z)\le 1$ gives
\[
\int_{\gamma'}\frac{d}{d\delta}\lambda_\delta \le (1+\epsilon_0) (\gamma'\cdot\Sigma)\frac{d\bar{\delta}}{d\delta}\int_0^{s_0}e^{\bar{\delta}\zeta(s)}\zeta(s)ds.
\]
On the other hand, it follows from differentiating \eqref{eqn:uniquedelta} that
\[
\frac{d\bar{\delta}}{d\delta}\int_0^{s_0}e^{\bar{\delta}\zeta(s)}\zeta(s)ds=1.
\]
Combining the above two lines proves property (c)(ii).

{\em Step 7.\/}
We have shown that the one-parameter family of contact forms $\{\lambda_\delta\}_{\delta\ge0}$ satisfies properties (a)--(c) for $\delta$ sufficiently small. We now show that the family $\{\lambda_\delta\}$ can be modified to also satisfy (d) for $\delta$ sufficiently small in the same sense.

By property (b) and Step 3, there is a number $r_0'<r_0$ such that for any degenerate Reeb orbit $\gamma'$ of $\lambda_\delta$, there exists a simple Reeb orbit $\gamma$ of $\lambda$ in $u(\partial\Sigma)$, such that $\gamma'$ intersects $N(\gamma)\setminus\mathcal{S}$ in the region where $\frac{1}{3}r_0 < r < r_0'$. By a standard transversality argument, we can perform a $C^\infty$-small perturbation of the family $\{\lambda_\delta\}$, replacing $\lambda_\delta$ by $e^{h_\delta}\lambda_\delta$ where $h_\delta:Y\to[0,\bar{\delta}/2]$ is supported in the union of the subsets of $N(\gamma)\setminus\mathcal{S}$ where $\frac{1}{3}r_0 < r < r_0'$, to arrange that $\lambda_\delta$ is nondegenerate for almost every $\delta$. We can further assume that $h_\delta$ is a monotone increasing function of $\delta$. Then since $\overline{\delta}$ is a monotone increasing function of $\delta$, it follows from \eqref{eqn:lambdadelta} that $e^{h_\delta}\lambda_\delta$ is obtained from $\lambda$ by multiplying by a monotone increasing function of $\delta$. Then by Remark~\ref{rem:altspecmonotone}, for each positive integer $k$, the spectral invariant $c_k(Y,\lambda_\delta)$ is a monotone increasing function of $\delta$, and hence differentiable almost everywhere. Thus property (d) holds. We claim that if the perturbation $\{h_\delta\}$ is sufficiently $C^1$-small, then properties (a)--(c) still hold.

Property (a) still holds by construction. Property (b) still holds because if $\gamma\in\mathcal{P}(\lambda)$, then $\imath_\delta(\gamma)$ does not intersect the support of the perturbation.  Property (c) still holds if $h_\delta$ is sufficiently $C^1$-small, by $C^1$-continuity of the estimates used to prove it.
\end{proof}

\begin{remark}
\label{rem:deltass}
The criterion for $\delta$ to be ``sufficiently small'' in Lemma~\ref{lem:slab} is of the form $\delta=O(r_0^2\epsilon_0^2)$. The constant in this bound depends continuously on the actions $\mathcal{A}(\gamma)$ and rotation numbers $\op{rot}_\Sigma(\gamma)$ for the simple Reeb orbits $\gamma\subset u(\partial\Sigma)$.
\end{remark}


\subsection{Behavior of the spectral invariants under inflation}

We now bound the growth of the elementary spectral invariants $c_k$ under a deformation as in Lemma~\ref{lem:slab}. In summary, an upper bound on the frequency of intersections of Reeb orbits with $\Sigma$, namely hypothesis \eqref{eqn:FL} below, leads to an upper bound on the ratio $c_k(Y,\lambda_\delta)/c_k(Y,\lambda)$, namely inequality \eqref{eqn:iti} below.

\begin{lemma}
\label{lem:slabderivative}
Let $u:\Sigma\to(Y,\lambda)$ be an admissible symplectic surface. Assume that $(Y,\lambda,u)$ is very nice and $\lambda$ is nondegenerate. Let
\begin{equation}
\label{eqn:Llb}
L \ge \max\left\{\mathcal{A}(\gamma)\;\big|\; \gamma\in\mathcal{P}(\lambda),\;\gamma\subset u(\partial\Sigma)\right\}.
\end{equation}
Let
\begin{equation}
\label{eqn:FL}
F \ge \sup\left\{\frac{\gamma\cdot\Sigma}{\mathcal{A}(\gamma)}\;\bigg|\;\gamma\in\mathcal{P}(\lambda),\;\mathcal{A}(\gamma) \le L\right\}.
\end{equation}
Let $\epsilon>0$ be given.
Fix $\epsilon_0>0$ such that
\begin{equation}
\label{eqn:epsilonprime}
(1+\epsilon_0)(F+\epsilon_0) < F+\epsilon.
\end{equation}
Let $\{\lambda_\delta\}_{\delta\ge0}$ be a family of contact forms provided by Lemma~\ref{lem:slab} for some $r_0,s_0>0$, and assume that $\delta>0$ is sufficiently small as in Lemma~\ref{lem:slab}. If $k$ is a positive integer, and if
\begin{equation}
\label{eqn:Lhyp}
c_k(Y,\lambda_{\delta}) \le L,
\end{equation}
then
\begin{equation}
\label{eqn:iti}
c_k(Y,\lambda_\delta) < e^{\delta (F+\epsilon)}c_k(Y,\lambda).
\end{equation}
\end{lemma}

\begin{proof}
Fix a positive integer $k$. Suppose that the contact form $\lambda_\delta$ is nondegenerate and that the function $\delta\mapsto c_k(Y,\lambda_\delta)$ is differentiable at $\delta$. By Lemma~\ref{lem:derivative}, there exists an orbit set $\alpha=\{(\alpha_i',m_i)\}$ for $\lambda_\delta$ such that
\begin{align}
\label{eqn:F1}
c_k(Y,\lambda_\delta) &= \sum_im_i\mathcal{A}_{\lambda_\delta}(\alpha_i'),\\
\label{eqn:F2}
\frac{d}{d\delta}c_k(Y,\lambda_\delta) &= \sum_im_i\int_{\alpha_i'}\frac{d}{d\delta}\lambda_\delta.
\end{align}
By Lemma~\ref{lem:slab}(b),(c), for each $i$ we have
\begin{equation}
\label{eqn:F3}
\int_{\alpha_i'}\frac{d}{d\delta}\lambda_\delta
\le
(1+\epsilon_0)
(\alpha_i'\cdot\Sigma).
\end{equation}
By Lemma~\ref{lem:slab}(b),(c), for each $i$ there is a simple Reeb orbit $\alpha_i\in\mathcal{P}(\lambda)$
such that
\begin{equation}
\label{eqn:ffai}
\frac{\alpha_i'\cdot\Sigma}{\mathcal{A}_{\lambda_\delta}(\alpha_i')} < \frac{\alpha_i\cdot\Sigma}{\mathcal{A}_{\lambda}(\alpha_i)} + \epsilon_0.
\end{equation}
If $\alpha_i$ is not in $u(\partial\Sigma)$, i.e.\ if we are in the situation of Lemma~\ref{lem:slab}(b), then $\mathcal{A}_\lambda(\alpha_i)\le \mathcal{A}_{\lambda_\delta}(\alpha_i')$, so by \eqref{eqn:Lhyp} and \eqref{eqn:F1},  we have $\mathcal{A}_\lambda(\alpha_i) \le L$. By this and the hypothesis \eqref{eqn:Llb}, for every $i$ we have $\mathcal{A}_\lambda(\alpha_i)\le L$, and then by the hypothesis \eqref{eqn:FL} we have
\begin{equation}
\label{eqn:fftaffai}
\frac{\alpha_i\cdot \Sigma}{\mathcal{A}_{\lambda}(\alpha_i)} \le F.
\end{equation}
It follows from \eqref{eqn:ffai} and \eqref{eqn:fftaffai} that
\begin{equation}
\label{eqn:F4}
\alpha_i'\cdot\Sigma \le (F+\epsilon_0)\mathcal{A}_{\lambda_\delta}(\alpha_i').
\end{equation}
By \eqref{eqn:F2}, \eqref{eqn:F3}, \eqref{eqn:F4}, \eqref{eqn:epsilonprime}, and \eqref{eqn:F1}, we have
\begin{equation}
\label{eqn:diac}
\frac{d}{d\delta}c_k(Y,\lambda_\delta) < (F+\epsilon)c_k(Y,\lambda_\delta).
\end{equation}
By Lemma~\ref{lem:slab}(d), the inequality 
\eqref{eqn:diac} holds for almost all $\delta\ge0$ which are sufficiently small as in Lemma~\ref{lem:slab}. Note that by the Conformality property in Proposition~\ref{prop:altspecproperties} and the monotonicity in Remark~\ref{rem:altspecmonotone}, the function $\delta\mapsto c_k(Y,\lambda_\delta)$ is locally Lipschitz, hence absolutely continuous on an interval. Thus we can integrate the differential inequality \eqref{eqn:diac} to obtain the desired inequality \eqref{eqn:iti}.
\end{proof}


\subsection{Proof of Theorem~\ref{thm:main1}}

Suppose $(Y,\lambda,u)$ is very nice as in Definition~\ref{def:verynice}. Let $\epsilon>0$ be given. Define
\begin{equation}
\label{eqn:finalf}
F = \frac{\op{Area}(\Sigma,d\lambda)}{\op{vol}(Y,\lambda)} - 3\epsilon.
\end{equation}
Choose $\epsilon_0>0$ satisfying \eqref{eqn:epsilonprime}. Choose $r_0$, $s_0$, and $\zeta$ as in \S\ref{sec:inflation}. Assume that $r_0$ is sufficiently small that
\begin{equation}
\label{eqn:finalr0}
\frac{\op{Area}(\Sigma_0,d\lambda)}{\op{vol}(Y,\lambda)} \ge \frac{\op{Area}(\Sigma,d\lambda)}{\op{vol}(Y,\lambda)} - \epsilon.
\end{equation}
Write $A_0=\op{Area}(\Sigma_0,d\lambda)$ and $V=\op{vol}(Y,\lambda)$. Note that by \eqref{eqn:finalf} and \eqref{eqn:finalr0}, we have
\begin{equation}
\label{eqn:feav}
F + \epsilon < \frac{A_0}{V}.
\end{equation}
Given $\delta>0$, define $\bar{\delta}$ by equation \eqref{eqn:uniquedelta}, and define a four-dimensional symplectic manifold
\begin{equation}
\label{eqn:Mssd}
M_{\Sigma_0,\delta} = \left\{(\sigma,s,z)\in\R\times[0,s_0]\times\Sigma_0\;\big|\; 0 < \sigma < \bar{\delta}\zeta(s)\right\} 
\end{equation}
with the symplectic form
\[
\omega = d(e^\sigma\lambda).
\]

\begin{lemma}
\label{lem:invokeweyl}
Suppose that $\delta>0$ is sufficiently small that
\begin{equation}
\label{eqn:dssh}
e^{2\delta (F+\epsilon)} < 1 + \frac{2A_0\delta}{V}.
\end{equation}
Then for any positive integer $N$, there exist positive integers $k,l>N$ such that
\begin{equation}
\label{eqn:kln}
\frac{c_k(Y,\lambda)+c_l^{\op{Alt}}(M_{\Sigma_0,\delta})}{c_{k+l}(Y,\lambda)} > e^{\delta (F+\epsilon)}.
\end{equation}
\end{lemma}

\begin{proof}
By perturbing $\delta$ slightly if necessary, we can assume without loss of generality that $A_0\delta/V$ is rational.

Write $M=M_{\Sigma_0,\delta}$. By equation \eqref{eqn:volmf}, we have
\[
\begin{split}
\op{vol}(M,\omega) &= \frac{1}{2}\int_{[0,s_0]\times\Sigma_0}\left(e^{2\bar{\delta}\zeta(s)}-1\right)\lambda\wedge d\lambda\\
&= \frac{A_0}{2}\int_0^{s_0}\left(e^{2\bar{\delta}\zeta(s)}-1\right)ds\\
&> A_0\int_{0}^{s_0}\left(e^{\bar{\delta}\zeta(s)}-1\right)ds\\
& =\delta A_0.
\end{split}
\] 
Here the second to last inequality holds by convexity of the exponential function, and the last equality holds by \eqref{eqn:uniquedelta}.
By the above volume estimate and the Weyl law for the alternative ECH capacities $c_l^{\op{Alt}}$ in Lemma~\ref{lem:sss}, we have
\begin{equation}
\label{eqn:limli}
\liminf_{l\to\infty}\frac{c_l^{\op{Alt}}(M,\omega)^2}{l} > 4A_0\delta.
\end{equation}

Now let $k$ and $l$ be positive integers such that
\begin{equation}
\label{eqn:lk}
\frac{l}{k} = \frac{2A_0\delta}{V}.
\end{equation}
By the Weyl law for the spectral invariants $c_k$ in Proposition~\ref{prop:altspecproperties}, we have
\begin{equation}
\label{eqn:limck}
\lim_{k\to\infty}\frac{c_k(Y,\lambda)^2}{k} = 2 V.
\end{equation}
It follows from equations \eqref{eqn:limli}, \eqref{eqn:lk}, and \eqref{eqn:limck} that
\begin{align}
\label{eqn:limckl}
\lim_{k\to\infty}\frac{c_{k+l}(Y,\lambda)^2}{k} &= 2V\left(1+\frac{2A_0\delta}{V}\right),\\
\label{eqn:clalt}
\liminf_{k\to\infty}\frac{c_l^{\op{Alt}}(M,\omega)^2}{k} &> \frac{8A_0^2\delta^2}{V}.
\end{align}
Combining \eqref{eqn:limck}, \eqref{eqn:limckl}, and \eqref{eqn:clalt} gives
\begin{equation}
\label{eqn:combinelimits}
\liminf_{k\to\infty}\frac{(c_{k}(Y,\lambda) + c_l^{\op{Alt}}(M,\omega))^2}{c_{k+l}(Y,\lambda)^2} > 1 + \frac{2A_0\delta}{V}.
\end{equation}
It follows from \eqref{eqn:dssh} and \eqref{eqn:combinelimits} that if $k$ and $l$ are sufficiently large, then \eqref{eqn:kln} holds.
\end{proof}

\begin{lemma}
\label{lem:nondegeneratemain}
With the above choices, suppose that $\lambda$ is nondegenerate. Suppose that $\delta>0$ is sufficiently small as in Lemmas~\ref{lem:slab} and \ref{lem:invokeweyl}. Let $k,l$ be positive integers provided by Lemma~\ref{lem:invokeweyl}, and assume these are sufficiently large that
\begin{equation}
\label{eqn:newklbig}
c_{k+l}(Y,\lambda) > \max\{\mathcal{A}(\gamma) \mid \gamma\in\mathcal{P}(\lambda), \gamma\subset u(\partial\Sigma)\}.
\end{equation}
Then there exists a simple Reeb orbit $\gamma\in\mathcal{P}(\lambda)$ with
\begin{equation}
\label{eqn:newintersectionbound}
\frac{\gamma\cdot\Sigma}{\mathcal{A}(\gamma)} \ge F
\end{equation}
and
\begin{equation}
\label{eqn:newactionbound}
\mathcal{A}(\gamma) \le e^{\bar{\delta}}c_{k+l}(Y,\lambda).
\end{equation}
\end{lemma}

\begin{proof}
By the Capacity Bound property \eqref{eqn:capacitybound} in Proposition~\ref{prop:altspecproperties} and the Monotonicity property \eqref{eqn:altechmonotonicity} in Lemma~\ref{lem:altechproperties}, we have
\[
c_{k+l}(Y,\lambda_\delta) \ge c_k(Y,\lambda) + c_l^{\op{Alt}}(M_{\Sigma_0,\delta}).
\]
Combining this with \eqref{eqn:kln} gives
\begin{equation}
\label{eqn:growfast}
\frac{c_{k+l}(Y,\lambda_\delta)}{c_{k+l}(Y,\lambda)} > e^{\delta (F+\epsilon)}.
\end{equation}
Now set
\begin{equation}
\label{eqn:newsetL}
L = c_{k+l}(Y,\lambda_\delta).
\end{equation}
By Lemma~\ref{lem:slab}(a) and the monotonicity of $c_{k+l}$ in Remark~\ref{rem:altspecmonotone}, we have
\begin{equation}
\label{eqn:newcklm}
c_{k+l}(Y,\lambda_\delta) \ge c_{k+l}(Y,\lambda).
\end{equation}
It follows from \eqref{eqn:newsetL}, \eqref{eqn:newcklm}, and \eqref{eqn:newklbig} that the inequality \eqref{eqn:Llb} holds. And it follows from \eqref{eqn:newsetL} that \eqref{eqn:Lhyp} holds with $k$ replaced by $k+l$. Then by Lemma~\ref{lem:slabderivative}, there exists a simple Reeb orbit $\gamma\in\mathcal{P}(\lambda)$ with
\begin{equation}
\label{eqn:newAgammaL}
\mathcal{A}(\gamma)\le L
\end{equation}
satisfying \eqref{eqn:newintersectionbound}. Since $\lambda_\delta/\lambda\le e^{\bar{\delta}}$ by Lemma~\ref{lem:slab}(a), it follows from the Conformality property in Proposition~\ref{prop:altspecproperties} and the Monotonicity property in Remark~\ref{rem:altspecmonotone} that
\[
c_{k+l}(Y,\lambda_\delta) \le e^{\bar{\delta}}c_{k+l}(Y,\lambda).
\]
Together with \eqref{eqn:newsetL} and \eqref{eqn:newAgammaL}, this proves \eqref{eqn:newactionbound}.
\end{proof}

\begin{proof}[Proof of Theorem~\ref{thm:main1}]
Let $\epsilon>0$ be given. Choose $F,r_0,s_0,\zeta$ as in Lemma~\ref{lem:invokeweyl}.

Suppose that $\lambda$ is nondegenerate. Fix $\delta>0$ sufficiently small as in Lemma~\ref{lem:slab} satisfying \eqref{eqn:dssh}. Let $k,l$ be positive integers provided by Lemma~\ref{lem:invokeweyl}, and assume these are sufficiently large that \eqref{eqn:newklbig} holds. Then by Lemma~\ref{lem:nondegeneratemain}, there exists a simple Reeb orbit $\gamma\in\mathcal{P}(\lambda)$ such that
\begin{equation}
\label{eqn:3epsilon}
\frac{\gamma\cdot\Sigma}{\mathcal{A}(\gamma)} \ge \frac{\op{Area}(\Sigma,d\lambda)}{\op{vol}(Y,\lambda)} - 3\epsilon
\end{equation}
and \eqref{eqn:newactionbound} holds.

Suppose now that $\lambda$ is degenerate. We can find a sequence of nondegenerate contact forms $(\lambda(n))_{n\ge 1}$ such that $\lambda(n)$ converges in $C^\infty$ to $\lambda$, and the triple $(Y,\lambda(n),u)$ is very nice for each $n$. By Remark~\ref{rem:deltass} 
we can choose $\delta>0$ which is sufficiently small as in Lemma~\ref{lem:slab} for all $n$ sufficiently large. By continuity of symplectic area and contact volume, we can also choose this $\delta$ to satisfy \eqref{eqn:dssh} for all $n$ sufficiently large. Let $k,l$ be positive integers provided by Lemma~\ref{lem:invokeweyl} for $\lambda$, so that the inequality \eqref{eqn:kln} holds, and assume that $k$ and $l$ are sufficiently large that \eqref{eqn:newklbig} holds for $\lambda$. By the continuity of alternative ECH capacities and elementary spectral invariants, if $n$ is sufficiently large, then the inequalities \eqref{eqn:kln} and \eqref{eqn:newklbig} also hold with $\lambda$ replaced by $\lambda(n)$. Then by the previous paragraph, for each $n$ sufficiently large we can find a simple Reeb orbit $\gamma(n)\in\mathcal{P}(\lambda(n))$ such that
\begin{equation}
\label{eqn:3epsilonn}
\frac{\gamma(n)\cdot\Sigma}{\mathcal{A}_{\lambda(n)}(\gamma(n))} \ge \frac{\op{Area}(\Sigma,d\lambda(n))}{\op{vol}(Y,\lambda(n))} - 3\epsilon
\end{equation}
and
\begin{equation}
\label{eqn:newactionboundn}
\mathcal{A}_{\lambda(n)}(\gamma(n)) \le e^{\bar{\delta}}(1+\epsilon)c_{k+l}(Y,\lambda(n)).
\end{equation}

It follows from \eqref{eqn:newactionboundn} and the continuity of elementary spectral invariants that there is an $n$-independent upper bound on the action $\mathcal{A}_{\lambda(n)}(\gamma(n))$. Consequently, we can pass to a subsequence so that the Reeb orbits $\gamma(n)$ converge to a Reeb orbit $\gamma$ for $\lambda$ satisfying \eqref{eqn:3epsilon}. It is possible that $\gamma$ is not a simple Reeb orbit, but in that case it is a multiple cover of a simple Reeb orbit which still satisfies \eqref{eqn:3epsilon}.

Since we have shown that given any $\epsilon>0$ there exists $\gamma\in\mathcal{P}(\lambda)$ satisfying \eqref{eqn:3epsilon}, this completes the proof of Theorem~\ref{thm:main1}.
\end{proof}


\end{document}